\documentclass{article}

\usepackage{amssymb, latexsym,pdfsync,amsmath,amsthm, ulem,xcolor}
\usepackage{comment}

  \usepackage[bookmarks=false]{hyperref}

\newtheorem{theorem}{Theorem}[section]
\newtheorem{corollary}[theorem]{Corollary}
\newtheorem{lemma}[theorem]{Lemma}
\newtheorem{proposition}[theorem]{Proposition}

\theoremstyle{definition}
\newtheorem{definition}[theorem]{Definition}
\newtheorem{remark}[theorem]{Remark}

\newtheorem{result}[theorem]{Result}

\newcommand{\FF}{\mathbb{F}}

\newcommand{\Fq}{\mathbb{F}_q}
\newcommand{\fq}{\mathbb{F}_q}
\newcommand{\Fqn}{\mathbb{F}_{q^n}}
\newcommand{\Fqm}{\mathbb{F}_{q^m}}

\newcommand{\Fqt}{\mathbb{F}_{q^2}}

\newcommand{\D}{\mathcal D}

\def\F{\mathbb{F}}
\def\Fq{{\mathbb{F}}_q}

\def\PG{\mathrm{PG}}
\def\V{\mathrm{V}}
\def\GL{\mathrm{GL}}

\def\dim{\mathrm{dim}}

\def\tr{\mathrm{tr}}

\def\trk{\mathrm{trk}}
\def\mrk{\mathrm{mrk}}

\def\cB{\mathcal B}
\def\cD{\mathcal D}

\def\cH{\mathcal H}

\def\cL{\mathcal L}

\def\cO{\mathcal O}

\def\cT{\mathcal T}

\def\cS{\mathcal S}

\def\cQ{\mathcal Q}
\def\cZ{\mathcal Z}

\def\tS{\Tilde{S}}
\def\tcS{\Tilde{\cS}}

\newcommand{\npmatrix}[1]{\left[ \begin{matrix} #1 \end{matrix} \right]}

\newcommand{\rank}{\mathrm{rank}}

\def\titlerunning#1{\gdef\titrun{#1}}
\makeatletter
\def\author#1{\gdef\autrun{\def\and{\unskip, }#1}\gdef\@author{#1}}

\makeatother

\def\MSC#1{{\renewcommand{\thefootnote}{}%
\footnote{\emph{Mathematics Subject Classification (2010):} #1}}}
\def\keywords#1{\par\medskip
\noindent\textbf{Keywords:} #1}

\usepackage{authblk}

\begin{document}

\baselineskip=16pt

\titlerunning{}
\title{On the geometry of tensor products over finite fields}
\author[1]{Stefano Lia}
\author[2]{John Sheekey}
\affil[1,2]{School of Mathematics and Statistics, University College Dublin, Dublin. E-mail: stefano.lia@ucd.ie, john.sheekey@ucd.ie}

\date{
}
\maketitle

\bigskip

\MSC{Primary 51E12, 
12K10, 
47A80, 
; Secondary
51E20, 
51A50
}

\begin{abstract}
In this paper we study finite dimensional algebras, in particular finite semifields, through their correspondence with nonsingular threefold tensors. We introduce a alternative embedding of the tensor product space into a projective space. This model allows us to understand tensors and their contractions in a new geometric way, relating the contraction of a tensor with a natural subspace of a subgeometry. This leads us to new results on invariants and classifications of tensors and algebras and on nonsingular fourfold tensors. A detailed study of the geometry of this setup for the case of the threefold tensor power of a vector space of dimension two over a finite field surprisingly leads to a new construction of quasi-Hermitian varieties in $\PG(3,q^2)$.
\keywords{Finite semifields, tensor representation of semifields, Quasi-Hermitian surface, two-character set.}
\end{abstract}

\section{Introduction}

In this paper we investigate the geometry of finite semifields of dimension $n$ over $\fq$ via their representation as threefold tensors of format $n\times n \times n$ and embedding of such tensors in projective spaces. While in the literature most work is focused on the standard embedding in the space $\PG(n^3-1,\fq)$, we work instead on the cyclic model, namely embedding $\fq^n\otimes\fq^n\otimes\fq^n$ in the space $\PG(n^2-1,q^n)$. This has some drawbacks, but we will show that a lot of interesting geometry arises. In particular, we are able to deal with a tensor and (one of) its contractions in the same space, see Section \ref{sec: contractions in the cyclic model}.
We study the case $n=2$ in greater detail. In this case we recover a new geometrical proof of Dickson classification of semifields of dimension two over their center \cite{dickson}. Furthermore, the nonsingularity of tensors in the cyclic model naturally provides a link with a pair of commuting polarities. The investigation of such polarities lead us to the construction of new quasi-Hermitian surfaces.

The paper is organised as follows: in Sections \ref{sec:one- and twofold} and \ref{sec:threefold tensors and algebras} we describe the standard and cyclic model for one-two and three-fold tensors. In Sections \ref{sec:contractions and nonsingularity} we define tensor contractions and in \ref{sec: contractions in the cyclic model} we illustrate their geometrical interpretation in the cyclic model. In Sections \ref{sec:groupaction}, and \ref{symmetricaction} we illustrate the group action relative to threefold tensors and in \ref{sec:thirdgl} we show how to interpret this in the cyclic model.
In Section \ref{sec:bound Bel rank} we exploit the cyclic model to obtain a bound on the BEL rank of a generic semifield.
In Section \ref{section nonsingular tensors geometry} we investigate more closely the case $n=2$. In this case we provide a detailed analysis of the geometry of certain quadrics and hermitian surfaces associated to the nonsingularity of tensors. This allows us to obtain results on the classifications of nonsingular $2\times2\times2\times2$ tensors in Section \ref{sec:fourfold tensors}. In Section \ref{sec:quasi hermitian} we investigate the orbits on points of $\PG(3,q^2)$ under the action of (a subgroup of) the stabiliser of the quadric $\cQ_0$. This is finalised in Section \ref{section plane intersection numbers} where we compute the intersection numbers of planes of $\PG(3,q^2)$ with the (join of) orbits computed in Section \ref{sec:quasi hermitian}.
Using this, in Section \ref{sec: new quasi hermitian}, we provide our construction of new quasi-Hermitian surfaces.
Finally, in Section \ref{section delta equation} we point out how to obtain explicit equations for the quasi-Hermitian varieties of $\PG(n,q^2)$ constructed in \cite{Pavese2015}, for any $n\geq3$.

\section{A model for threefold tensors}\label{sec:model for threefold tensors}
In this section we develop a useful model for threefold tensors. We begin by recalling the well-known models for one-fold tensors (i.e. vectors) and two-fold tensors (i.e. matrices), before moving on to threefold tensors. All of this is well known in various formats, but we wish to formally introduce everything in a coherent setting.

\subsection{One- and twofold tensors}\label{sec:one- and twofold}

Let $V=V(n,q)$ be an $n$-dimensional vector space over the finite field $\Fq$. We can identify the elements of $V$ with the elements of the extension field $\Fqn$. We denote by $V^\vee$ the usual dual of $V$, that is, the space of linear forms from $V$ to $\Fq$.

It is well known that we can embed $\Fqn$ as a subspace of $V(n,q^n)$ in a number of ways. In the {\it standard model}, we choose an $\Fq$-basis $\{e_1,\ldots,e_n\}$ for $\Fqn$, and define a map $\phi_s:x = \sum_i x_ie_i\mapsto (x_1,x_2,\cdots,x_n)$. We denote the image of this embedding by $V_s$. The endomorphism ring for the standard model is the space of $n\times n$ matrices with entries in $\Fq$, which we denote as usual by $M_n(\Fq)$. 

Another convenient method is to use the {\it cyclic model}, which is associated with the map $\phi_c:x\mapsto (x,x^q,\cdots,x^{q^{n-1}})$. We denote the image of this embedding by $V_c$. The endomorphism ring for the cyclic model is the space of $n\times n$ {\it Dickson matrices} with entries in $\Fqn$, which we denote by $\cD_n(\Fq)$. This is an $\Fq$-subspace of $M_n(\Fqn)$, consisting of all matrices of the form
\[
\npmatrix{f_0&f_1&\ldots&f_{n-1}\\
		f_{n-1}^{q}&f_0^{q}&\ldots&f_{n-2}^q\\
		\vdots&\vdots&\ddots&\vdots\\
		f_1^{q^{n-1}}&f_2^{q^{n-1}}&\ldots&f_0^{q^{n-1}}}^T
\]
where $f_i\in \Fqn$. 
A matrix of this form acts on vectors in the cyclic model as $\phi_c(x)\mapsto \phi_c(f(x))$, where $f(x) = \sum_{i=0}^{n-1}f_ix^{q^i}$. We denote the above matrix by $D_f$. We can view $f(x)$ as a polynomial in $\Fqn[x]$; a polynomial of this shape is called a {\it linearised polynomial}. The group of isomorphisms for the cyclic model is then equal to the set of invertible Dickson matrices, and is isomorphic to $\GL(n,q)$.

The dual space of the standard model $V_s$ is obtained via the usual dot product; $\tilde{y}(x) := \sum x_iy_i$, where $x_i,y_i$ are the coordinates of $x$ and $y$ with respect to the same fixed basis. The dual space of the cyclic model $V_c$ is obtained via the trace map $\tr$ from $\Fqn$ to $\Fq$: $\hat{y}(x) := \tr(xy)= \sum_{i=0}^{n-1}(xy)^{q^i}$.

The tensor product space $\Fqn\otimes _{\Fq}\Fqn^\vee$, which is isomorphic to the space of $\Fq$-linear maps from $\Fqn$ to itself, can then be identified with either the space of matrices $M_n(\Fq)$ using the standard model $V_s\otimes V_s^\vee$, or the space of Dickson matrices $\cD_n(\Fq)$ using the cyclic model $V_c\otimes V_c^\vee$. Both of these are $\Fq$-subspaces of the space $M_n(\Fqn)$. These two spaces can be mapped to one another via conjugation by a {\it Moore matrix}.

The representation of linear maps by Dickson matrices, or the closely related concept of {\it linearised polynomials}, has often been used fruitfully to construct objects which appear much more difficult when working in $M_n(\Fq)$; for example, when constructing {\it MRD codes}, see \cite{SheekeyMRD}.

The image of the set of {\it pure tensors} is the set of rank one matrices in $M_n(\Fq)$ and $\cD_n(\Fq)$ respectively in the standard and cyclic models. The natural action of $\GL(n,q)^2$ on $\Fqn\otimes _{\Fq}\Fqn^\vee$ then acts via multiplication on the left and right by an invertible matrix in $M_n(\Fq)$ or $\cD_n(\Fq)$ respectively. In both cases, the natural action of $S_2$ on twofold tensors acts by transposing the matrices.

\subsection{Threefold tensors and algebras}\label{sec:threefold tensors and algebras}

Consider now a threefold tensor product space of $n$-dimensional vector spaces over $\Fq$. It is well-known that this space is isomorphic to the space of $\Fq$-bilinear maps from the product of two $n$-dimensional vector spaces to a third; that is, the space of (not necessarily associative) $n$-dimensional algebras over $\Fq$.

Using the standard model for $\Fqn$ to construct the tensor product space $V_s^\vee\otimes V_s^\vee\otimes V_s$, we obtain a correspondence between algebras and $n\times n\times n$ hypercubes of elements of $\Fq$, which can be represented by elements of $V(n^3,q)$, or projectively as points of $\PG(n^3-1,q)$. 

In this paper we will work instead with $V_c^\vee\otimes V_c^\vee\otimes \Fqn$, where all tensor products are over $\Fq$. In this model, pure tensors define $\Fq$-bilinear maps from $V_c\times V_c$ to $\Fqn$, which we can view also as $\Fq$-bilinear maps from $\Fqn\times \Fqn$ to $\Fqn$, in the following way:
\begin{align*}
\hat{\alpha}\otimes\hat{\beta}\otimes \gamma:(x,y)&\mapsto \tr(\alpha x)\tr(\beta y)\gamma \\
&=\phi_c(x)(\gamma \phi_c(\alpha)^T\phi_c(\beta))\phi_c(y)^T.
\end{align*}

Thus we can represent this map as a matrix $$\gamma\phi_c(\alpha)^T\phi_c(\beta) = \gamma \npmatrix{\alpha\\\alpha^q\\\vdots\\\alpha^{q^{n-1}}}
\npmatrix{\beta&\beta^q&\cdots&\beta^{q^{n-1}}}
\in M_n(\Fqn).$$ 

As usual, this can be extended linearly so that any element of $V_c^\vee\otimes V_c^\vee\otimes \Fqn$ defines an $\Fq$-bilinear map. From the additivity of the above expression, we get that every threefold tensor can be identified with a matrix $C\in M_n(\Fqn)$. Conversely, every matrix $C\in M_n(\Fqn)$ defines an $\Fq$-bilinear map in the following way:
\begin{align*}
C:(x,y) 
&\mapsto\phi_c(x)C\phi_c(y)^T\\
&= \sum_{i.j=0}^{n-1}c_{ij}x^{q^i}y^{q^j}.
\end{align*}
Hence we can identify the space of threefold tensors with the space $M_n(\Fqn)$, or projectively as points of $PG(n^2-1,q^n)$. As we will note later, moving to the projective setting has some drawbacks, but we will aim to show that some rich and interesting geometry arises. We will abuse terminology a little, and refer to this representation as the {\it cyclic model for threefold tensors}.
We denote by $\Phi_c$ the map corresponding to the cyclic representation of threefold tensors, namely 
\[
\Phi_c:V^\vee\otimes V^\vee\otimes\Fqn\mapsto M_n(\Fqn)
\]
with 
\[
\Phi_c(T)=\sum_{i=1}^r\gamma_i\phi_c(\alpha_i)^T\phi_c(\beta_i)
\]
and $T=\sum_{i=1}^r\alpha_i\otimes\beta_i\otimes\gamma_i$. We denote by $M_T$ the matrix representation $\Phi_c(T)$ of the tensor $T$. This representation of a multiplication is one of the commonly used representations in the theory of finite nonassociative algebras, particularly {\it finite semifields}; that is, nonassociative division algebras. However when studying the tensors associated to semifields, most work (e.g. \cite{LaShOrbits}\cite{LavrauwSheekey2022}), as well as most work on classifying tensors in general (e.g. \cite{Stavrou}) has focused on the standard model. This is in part because the easiest way to study threefold tensors is via their {\it contractions}, and the geometry of contractions in the standard model has to date been much easier. In this paper we will address this apparent weakness of the cyclic model. Indeed, we will see that we will be able to view a tensor and its contraction in the same space, related by natural geometric properties. 

Moreover, the geometry of this setup will lead us to further understanding of invariants of semifields and higher order nonsingular tensors. As a further unexpected consequence, through a detailed study of the orbits of points under a naturally occurring group action, we will obtain new examples of quasi-Hermitian varieties with automorphism group $\mathrm{PCGO}^+(3,q)$, and therefore new non-isomorphic strongly regular graph with the same automorphism group.

\subsection{Contractions and Nonsingularity}\label{sec:contractions and nonsingularity}

A {\it contraction} of a tensor $T\in \bigotimes_{i=1}^k V_i$ is the image of $T$ under an element of the dual space of one of the component vector spaces $V_i$. More precisely, given an element $v_j^\vee\in V_j^\vee$ and a pure tensor $T=\otimes_{i=1}^k u_i$, we define 
\[
v_j^\vee(T) = v_j^\vee(u_j)\bigotimes_{i=1,i\ne j}^k u_i\in \bigotimes_{i=1,i\ne j}^k V_i.
\]
This is then extended linearly to define the contraction of any tensor.

The contraction of a twofold tensor (or matrix, or linear map) is a onefold tensor (simply, a vector). This can be viewed as the evaluation of a linear map at a vector, or a linear combination of the rows or columns of a matrix. 

A tensor is {\it nonsingular} if every nonzero contraction of it is nonsingular; this gives a recursive definition, with the base case of a vector being nonsingular if and only if it is nonzero. Thus we see that in the twofold case this coincides with the usual definition of nonsingularity of a linear map or matrix; every nontrivial evaluation of the linear map is nonzero, and every nontrivial linear combination of the rows or columns of the matrix is nonzero. In the threefold case, a tensor is nonsingular precisely when the algebra it defines is a semifield. We refer to \cite{Lavrauw2013} for background.

A {\it contraction space} of a tensor is the space obtained by contracting by all vectors in a fixed dual space $V_j^\vee$; that is, $C_j(T) = \{v_j^\vee(T):v_j^\vee\in V_j^\vee\}$.
 A contraction space is said to be nonsingular if all its elements are nonsingular, and its dimension is equal to $\dim(V_j)$.
A contraction space is a subspace of a tensor product space with fewer factors than $T$; as such, they are not naturally defined in the same space, and so it is difficult to analyse tensors and their contractions in the same setting. However we will now demonstrate how we can place both a tensor and its contraction in the same setting.

In particular, one contraction space of a tensor in the cyclic model $V_c^\vee\otimes V_c^\vee\otimes \Fqn$, which can be represented by $M_n(\Fqn)$, will naturally lie in $V_c^\vee\otimes V_c^\vee$, which can be represented by $\cD_n(\Fq)$, an $\Fq$-subspace of $M_n(\Fqn)$. Projectively, the points of $\PG(n^2-1,q^n)$ defined by Dickson matrices form a subgeometry $\Sigma$ isomorphic to $\PG(n^2-1,q)$.


\subsection{Group actions, isotopisms, and collineations}\label{sec:groupaction}
The group $\mathrm{GL}(V)\times \mathrm{GL}(V)\times \mathrm{GL}(V)$ naturally acts on $V^\vee\otimes V^\vee\otimes V$. The action of $(f,g,h)\in\mathrm{GL}(V)\times \mathrm{GL}(V)\times \mathrm{GL}(V)$ is defined on pure tensors as
\[
(a^{\vee} \otimes b^{\vee} \otimes c)^{(f,g,h)} := f(a)^{\vee} \otimes g(b)^{\vee} \otimes h(c);
\]
and extended linearly to all tensors. The induced action on the algebras corresponding to threefold tensors coincides with the notion of {\it isotopism}; two muliplications $\circ$ and $\star$ are {\it isotopic} if there exist invertible additive maps $(f,g,h)$ such that $h(f(x)\circ g(y))= x\star y$.

The action induced on $M_n(\Fqn)$ can be seen as follows. We split the description into two components for convenience.
\begin{align*}
    (f,g,1)&:M_T\mapsto D_f M_TD_g^t,\\
    (1,1,h)&:M_T\mapsto \left(\sum_{i=0}^{n-1}h_i M_T^{\sigma^i}\right)
\end{align*}
where $D_f,D_g$ are the Dickson matrices corresponding to the maps $f,g$ respectively, and $M_T^\sigma$ is the matrix whose $(i,j)$-entry is $\alpha_{i-1,j-i}^q$, where $\alpha_{i,j}$ is the $(i,j)$-entry of $M_T$, and indices are taken modulo $n$.

Note that the maps $(f,g,h)$ define a collineation of $\PG(n^2-1,q^n)$ if and only if $h$ is a monomial; that is, $h(x)= ax^{q^k}$ for some $a\in \Fqn$ and some integer $k$. However, in Section \ref{sec:thirdgl} we will see that the action of $GL(V)$ on the third factor can still be interpreted in a geometric way.

We recall two relevant invariants for tensors and their related algebras. 

\begin{definition}
    The {\it matrix rank} $\mrk(T)$ of $T$ is defined as the minimum rank of the matrix $M_S$ representing a tensor $S$ equivalent to $T$ in the cyclic model.
\end{definition}

We will abuse notation and denote by $T(x,y)$ the bilinear multiplication associated with the tensor $T$.
Then we have $T^{(f,g,h)}(x,y) = \left( T(x^f,y^g)\right)^h$.

The {\it tensor rank} $\trk(T)$, or simply the rank, of the tensor $T\in\bigotimes_{i=1}^n\V_i$ is the minimum integer $r$ such that we can write 
$T$ as sum of $r$ pure tensors,
\[
T=\sum_{j=1}^rv_1^j\otimes\cdots\otimes v_n^j.
\]
Clearly $\trk(T) = \trk(T^{(f,g,h)})$ for all $f,g,h \in \GL(n,q)$.
The matrix rank $\mrk(T)$ of $T$ is defined as the rank of the matrix $M_T$ representing $T$ in the cyclic model.
We have that 
\[
\trk(T)\geq\mrk(T)
\]
for all $T$.


\subsection{Action of $Sym(3)$}\label{symmetricaction}

Threefold tensors are usually studied up to $\mathrm{GL}(\V)\wr\mathrm{Sym}(3)$-equivalence, with the natural action of $S_3$ on threefold tensors. 


The action of $S_3$ on the cyclic model can be realised as follows. Let $\mathrm{Sym}(3)=\langle \tau_1,\tau_2\rangle$, where $\tau_1=(12)$ and $\tau_2=(123)$.
Then $\tau_1$ acts as a transposition, namely $\tau_1(M_T)_{i,j}=(M_T)_{j,i}$, and $\tau_2$ acts on $M_T$ via
\[
\tau_2(M_T)_{i,j}=\big( (M_T)_{j-i,-i}\big)^{q^i},
\]
where indices are taken modulo $n$. To see this, consider the action of $\tau_1$ and $\tau_2$ on pure tensors. We have $\tau_1(a^\vee\otimes b^\vee\otimes c)=(b^\vee\otimes a^\vee\otimes c)$, and $\tau_2(a^\vee\otimes b^\vee\otimes c)=(c^\vee\otimes a^\vee\otimes b)$. The claim for pure tensors is then straightforward by looking at the corresponding matrices.
Since the action is $\fq$-linear it holds for general tensors as well.

\subsection{Finalising our setup}\label{sec:finalising}
Throughout the remainder of the paper we will work with the following objects with the following notation.
\begin{itemize}
    \item The space $\Sigma^*=\PG(n^2-1,q^n)$, with points $P= \langle M_T\rangle$ for $M_T=(c_{ij})_{i,j}\in M_n(\Fqn)$. {Abusing notation, we will write $P=\langle T \rangle$ as well.}
\item The collineation $\sigma$ of $\Sigma^*$ defined by the semilinear map (also denoted $\sigma$, with some abuse of notation) mapping $M_T$ to $M_T^\sigma$ where $(M_T^\sigma)_{i,j} = c_{i-1,j-1}^q$.
    
    \item The subgeometry $\Sigma$ of points fixed by $\sigma$, consisting of points defined by Dickson matrices in $\D_n(\Fq)$. 
\end{itemize}

\subsection{Contractions in the cyclic model}\label{sec: contractions in the cyclic model}

This setup now allows us to view the contraction space of a tensor and the tensor itself as two objects in the same projective space with a natural geometric relationship. This will be key to our results throughout the paper.

Though there are multiple choices for a contraction space, it turns out the one that fits most conveniently into this setting is the third contraction space.


    


Consider a tensor $T =\sum_{i=1}^r u_i^\vee\otimes v_i^\vee\otimes w_i$, with corresponding matrix $M_T=\Phi_c(T)$. Let $P=\langle M_T\rangle$ denote the corresponding point of $\Sigma^*$. 

We can identify the contraction space $C_3(T)$ with a subspace of $\cD_n(\Fq)$, or projectively as a subspace of $\Sigma$, as follows. 

The contraction with $z^\vee\in V^\vee$, the dual space of the third vector space, is given by 
\[
z^\vee(T)=\sum_{i=1}^r u_i^\vee\otimes v_i^\vee z^\vee(w_i)=\sum_{i=1}^r u_i^\vee\otimes v_i^\vee \tr(zw_i).
\]
Note that, since $\tr(zw_i)\in\fq$, $z^\vee(T)$ can be identified with the matrix 
\[
z^\vee(M_T) := \sum_{i=1}^r \phi_c(u_i)^\mathrm{T}\phi_c(v_i)\tr(zw_i),
\]
which is a Dickson matrix. We denote the induced subspace of $\Sigma$ by $C_3(M_T)$.

Hence we can now see a tensor and its contraction space in the same projective space. In the next theorem, we show that in fact these obey by a natural geometric relationship.

\begin{theorem}\label{contraction is subline}
    The contraction space $C_3(M_T)$ of a tensor with matrix representation $M_T$ is the unique subspace of $\Sigma$ of minimal dimension whose extension to $\Sigma^*$ contains $P=\langle M_T\rangle$.
\end{theorem}

\begin{proof}
Let $T= \sum_{i=1}^r u_i^\vee\otimes v_i^\vee\otimes w_i$, with corresponding matrix $M_T$. The subspace of $\Sigma$ of minimal dimension whose extension to $\Sigma^*$ contains $P=\langle M_T\rangle$ is the space defined by the matrices
\[
\left\{\sum_j z^{q^j} M_T^{\sigma^j}:z\in \Fqm\right\}.
\]
The extension of this space to $\Sigma$ is the space $\langle M_T,M_T^\sigma,\ldots,M_T^{\sigma^{n-1}}\rangle$. Note that the (vector space) dimension of this space may be smaller than $m$. It remains to show that this space coincides with $C_3(M_T)$.


For $z\in \Fqn$ we have 
\begin{align*}
z^\vee(M_T) &= \sum_i \tr(zw_i)(\phi_c(u_i))^T\phi_c(v_i)\\
&=\sum_{i}\sum_j z^{q^j}w_i^{q^j}((\phi_c(u_i))^T\phi_c(v_i))\\
    &= \sum_{j} z^{q^j}\sum_i w_i^{q^j}((\phi_c(u_i))^T\phi_c(v_i))\\
    &= \sum_j z^{q^j} M_T^{\sigma^j}\in \mathrm{Fix}(\sigma),
\end{align*}
as required. Here we used that $M_T^{\sigma^j}= \sum_{i=1}^r \phi_c(u_i)^T\phi_c(v_i)(w_i^{q^j})$.
\end{proof}

For example, when $n=2$ we have that $\Sigma$ is a Baer subgeometry of $\Sigma^*$. It is well-known that every point of $\Sigma^*$ outside $\Sigma$ lies on a unique line meeting $\Sigma$ in a line of $\Sigma$.

We will denote the extension of the contraction space $C_3(M_T)$ by
\[
\overline{C}_3(M_T) := \langle M_T,M_T^\sigma,\ldots,M_T^{\sigma^{n-1}}\rangle.
\]

Note that we can see how this contraction space relates to the linear maps of multiplication in an algebra defined by $T$ as follows. Let $M_T= (\alpha_{ij})$. Then a Dickson matrix corresponding to $z^\vee(T)$ is $\sum_i z^{q^j}M_T^{\sigma^j}$, which corresponds to the linear map
\[
x\mapsto \sum_{i,j} \alpha_{-j,i-j}^{q^j}x^{q^i}z^{q^j} =:h_z(x).
\]
This is the map of multiplication on the right by $z$ in the algebra defined by $\tau_2^2(T)$ as defined in Section \ref{symmetricaction}.

Note that we can also see the relationship between the maps $h_z$ and the linear maps $f_y$ of right multiplication in the algebra defined by $T$ through the observation
\[
Tr(f_y(x)z)=Tr(h_z(x)y).
\]
This type of relationship naturally arises when viewing the action of $S_3$ on trilinear forms defined by tensors, rather than on the bilinear maps.

\subsection{Interpreting the third action of $GL(V)$ in the cyclic model}\label{sec:thirdgl}

In Section \ref{sec:groupaction} we saw that only two of the actions of $\GL(V)$ give rise to collineations of $\PG(n^2-1,q^n)$. However we can now give a geometric interpretation of the third action.

\begin{lemma}
    Let $h$ be an invertible $\Fq$-linear map from $\Fqn$ to itself, and let $T\in \Sigma^*=\PG(n^2-1,q^n)$. Then $T^{(1,1,h)}\in \overline{C}_3(T)$, and  $C_3(T^{(1,1,h)})= C_3(T)$.
\end{lemma}
\begin{proof}
    Let $h(x) = \sum_{i=0}^{n-1}h_ix^{q^i}$, and let $z\in \Fqn$. Then 
    \[
    M_{T^{(1,1,h)}} = \sum_{i=0}^{n-1}h_i M_T^{\sigma^i}\in \overline{C}_3(T).
    \]
    Since $C_3(T)$ is closed under $\sigma$, we have that $C_3(T^{(1,1,h)})\leq C_3(T)$. Since $h$ is invertible, we can apply the same argument to show that $ C_3(T)=C_3((T^{(1,1,h))})^{(1,1,h^{-1})}) \leq C_3(T^{(1,1,h)})$, proving the claim.
\end{proof}
Thus the third action of $\GL(V)$ fixes the contraction space, and permutes the points in the extension of this space. For example, when $n=2$, this group acts transitively on points of an extended subline of $\Sigma$ outside of $\Sigma$.

We now characterise which elements of $\overline{C}_3(T)$ represent tensors equivalent to $T$. For any set $S$ of points, denote by $\Omega_k(S)$ the set of points contained in some subspace spanned by $k+1$ points of $S$. We can refer to this as the {\it $k$-th secant variety} of $S$.

\begin{lemma}
    Suppose $\{M_T^{\sigma^i}:i=0,1,\ldots,n-1\}$ are linearly independent over $\Fqn$. Then the points of $\overline{C}_3(T)$ corresponding to tensors $T^{(1,1,h))}$ for some $h\in \GL(V)$ are precisely the points of $\overline{C}_3(T)$ not contained in $\Omega_{n-2}(C_3(T))$.
\end{lemma}

\begin{proof}
    From the proof of the previous lemma, we have that the set of points are those whose coordinates with respect to the basis $\{M_T^{\sigma^i}:i=0,1,\ldots,n-1\}$ are the coefficients of an invertible linearised polynomial $h(x)$. The points of $C_3(T)$ are those whose coordinates are of the form $(\alpha,\alpha^q,\ldots,\alpha^{q^{n-1}})$ for $\alpha\in \Fqn^\times$. These are precisely the coordinates of a linearised polynomial of rank one. Hence the result follows from the usual characterisation of invertible maps as those not contained in the $(n-2)$-nd secant variety to the set of rank one linear maps.
\end{proof}

\section{An upper-bound for the BEL rank of a semifield}\label{sec:bound Bel rank}

The {\it BEL rank} of a finite algebra is an isotopy invariant defined in \cite{LavrauwSheekey2017BEL}, and further studied in \cite{ZiniZullo2021}, \cite{SheekeyVandeVoordeVoloch2022}. It was shown in \cite[Theorem 11]{LavrauwSheekey2017BEL} that the BEL rank of an algebra defined by a tensor $T$ is equal to the
matrix rank of a tensor $S_3$-equivalent to $T$. In short, it is the minimal $r$ such that there exist linear maps $f_i,g_i$ such that the multiplication in one of the algebras $S_3$-equivalent to $T$ can be expressed as
\[
x\circ y = \sum_{i=1}^r f_i(x)g_i(y).
\]

A semifield has BEL-rank two if and only if it is isotopic to a field. Some important families, such as generalised twisted fields and semifields two-dimensional over a nucleus, are known to have BEL-rank two. Menichetti's celebrated classification result showed that if $n=3$, or if $n$ is prime and $q$ is large enough, then every semifield is isotopic to a field or a twisted field. 

In \cite{LavrauwSheekey2017BEL}, computational data for the case $q=2$, $n=6$ lead to the conjecture that the BEL-rank of a semifield is at most $n-1$ (and semifields with this BEL-rank do exist for these parameters, and other small parameters).

For a point of $\PG(n^2-1,q^n)$ in the cyclic model, we define its rank to be the usual linear algebraic rank of a matrix representing that point. We then obtain the following.

\begin{lemma}
    The BEL-rank of a tensor $T$
    is equal to the minimum rank of an element of $\overline{C}_3(T)$ not contained in $\Omega_{n-2}(C_3(T))$.
\end{lemma}

\begin{proof}
    First note that for any invertible $\Fq$-linear maps $f,g$, we have that $\rank(M_{T^{(f,g,1)}})= \rank(M_T)$. Hence to calculate the minimal rank of a matrix equivalent to $M_T$, it suffices to calculate 
    \[
    \min\{\rank(M_{T^{(1,1,h)}}):h\in \GL(V)\}.
    \]
    In Section \ref{sec:thirdgl}, we saw that the elements $M_{T^{(1,1,h)}}$, $h$ invertible, are precisely the elements of $\overline{C}_3(T)$ not contained in $\Omega_{n-2}(C_3(T))$.
\end{proof}

This leads us to a proof of the conjecture from \cite{LavrauwSheekey2017BEL}. The proof is in the same spirit as that of \cite[Theorem 12]{GowSheekey}.

\begin{theorem}
    The BEL-rank of an $n$-dimensional finite semifield over $\Fq$ is at most $n-1$, provided that $q$ is large enough with respect to $n$.
\end{theorem}

\begin{proof}
Let $T$ be a nonsingular tensor representing the semifield.
    Consider the polynomial
    \[
    F(x_1,\ldots,x_n)= \det\left(\sum_{i=0}^{n-1}x_iM_T^{\sigma^i}\right).
    \]
We need to show that there exists a nontrivial zero of this polynomial over $\Fqn$.

    Since $C_3(T)$ is spanned by Dickson matrices, and the determinant of a Dickson matrix lies in $\Fq$, we get that $F$ is a (scalar multiple of a) polynomial in $\Fq[x_1,\ldots,x_n]$. Moreover, $F$ does not possess any nontrivial zeroes of the form $(\alpha,\alpha^q,\ldots,\alpha^{q^{n-1}})$, for otherwise $C_3(T)$ would possess a nonsingular element, contradicting the nonsingularity of the tensor $T$. Hence $F$ is irreducible over $\Fq$.

    Let us consider the factorisation of $F$ over the algebraic closure of $\Fq$. We must have that $F$ possesses an absolutely irreducible factor $G$ of degree $n/s$ defined over a field $\FF_{q^s}$ for some divisor $s$ of $n$. If $s>1$, then by the Chevalley-Warning theorem $F$ would have a nontrivial zero over $\Fqn$, since the degree of $G$ is smaller than the number of variables, and the result follows for any $q$. 

    Suppose finally that $s=1$, that is, $F$ is absolutely irreducible. Then by the effective version of the Lang-Weil bound from \cite{CafureMatera}, $F$ will possess a nontrivial zero over $\Fqn$ provided 
    \[
    q^{n(n-1)} > (n-1)(n-2)q^{n(n-3/2)}+5n^{13/3}q^{n(n-2)}
    \]
    Hence the result follows whenever $q$ is large enough to satisfy this bound.
\end{proof}

Note that the range of applicability of this bound is large; for example, for $n=5$ we require only that $q>11$. Compare this to the result of Menichetti, which says that every semifield of dimension $5$ is equivalent to a field or twisted field (and hence has BEL-rank at most two) provided $q>6296$.

This result also gives a new proof of Dickson's result that every $2$-dimensional semifield is isotopic to a field; a semifield has BEL-rank one if and only if it is isotopic to a field. This in fact follows without appealing to the Lang-Weil bound; all that we require is that every extended subline of $\Sigma$ meets the quadric $\alpha\delta-\beta\gamma=0$ in at least a point. 




\section{The geometry of nonsinglar tensors for $n=2$}\label{section nonsingular tensors geometry}

We now turn our attention to a more detailed study of the nonsingular tensors when $n=2$. This case was studied in detail in the standard model in \cite{LavrauwSheekey2014}. It has been known since Dickson's result that all nonsingular tensors in this case are equivalent; however a deeper understanding of the geometry of this case is vital both for studying higher dimensions, and for studying nonsingular fourfold tensors. As we will see, some very interesting varieties also surprisingly arise from this detailed study.

As shown in the previous sections, it turns out that the most natural contraction to study in the cyclic model is $\tau_2^2(M_T)$. We write
\[
M_T= \npmatrix{\alpha&\beta\\ \gamma&\delta}.
\]
The Dickson matrices in the contraction space are represented by
\[
z^\vee(M_T) = \npmatrix{\alpha z+\delta^q z^q&\gamma z+\beta^qz^q\\ \gamma^qz^q+\beta z&\alpha^q z^q+\delta z}
\]
with corresponding linear maps  $h_z(x)=(\alpha z+\delta^q z^q)x+(\gamma z+\beta^qz^q)x^q$. 
The function $h_z(x)=(\alpha z+\delta^q z^q)x+(\gamma z+\beta^qz^q)x^q$ is nonsingular if and only if $h_z(x)\neq 0$ for $z\neq 0$. Therefore $M_T$ is nonsingular if and only if this holds for any $z\neq 0$, namely this is to say that, $\forall z,x\neq 0$, $(\alpha z+\delta^q z^q)x+(\gamma z+\beta^qz^q)x^q \neq 0$, namely  $-\frac{\alpha z+\delta^q z^q}{\gamma z+\beta^qz^q}\neq x^{q-1}$.
 Since $\{ z^{q-1}| z \in \mathbb{F}_{q^2}\setminus \{0\}\}=\{z\in\mathbb{F}_{q^2}|z^{q+1}=1\}$, this is equivalent to $(\alpha z+\delta^q z^q)^{q+1}\neq(\gamma z+\beta^qz^q)^{q+1}$, $\forall z \neq 0$.

 Expanding and reordering, this is equivalent to 
\begin{equation}
\begin{cases}\label{system}
&Q+H z+Q^q z^2\neq 0 \\
&z^{q+1}=1;
\end{cases}
\end{equation}
where $Q=\alpha\delta-\beta\gamma$ and $H=\alpha^{q+1}-\beta^{q+1}-\gamma^{q+1}+\delta^{q+1}$.
\subsection{Analysis of the System \eqref{system}}
The forms $Q$ and $H$ define respectively an hyperbolic quadric $\cQ^+$ and a non-degenerate Hermitian surface $\cH$.

If $Q=H=0$, namely $P\in\cQ^+\cap\cH$, then this system clearly has solutions, and so the tensor is singular. If $Q=0$ and $H\ne 0$, (namely $P\in\cQ^+\setminus\cH$) then the system has no solutions and the tensor is nonsingular.
Namely
\begin{itemize}
\item if $P\in\cQ^+\cap\cH$ then $P$ is singular;
\item if $P\in\cQ^+\setminus\cH$, then $P$ in nonsingular.
\end{itemize}
Assume now that $Q\ne 0$, namely $P\not\in\cQ^+$.
Let $\Delta=H^2-4Q^{q+1}$ and $\cD$ the associated variety. Then $\Delta(P)\in\mathbb{F}_q$ for any $P\in PG(3,q^2)$. Assume that $\Delta(P)\notin\square_q^\times$. 
Then 
\[\sqrt{\Delta(P)}^q=-\sqrt{\Delta(P)}\] 
and 
\[z^{q+1}=\frac{H(P)^2-\Delta(P)}{4Q(P)^{q+1}}=\frac{H(P)^2-H(P)^2+4Q(P)^{q+1}}{4Q(P)^{q+1}}=1.\]
Therefore, the system has two solutions, and the tensor is singular. Similarly, if $\Delta(P)=0$, then $z^{q+1}=\frac{H(P)^2}{4Q(P)^{q+1}}=1$, and so the system has a unique solution, and the tensor is singular.

On the other hand, if $\Delta(P)\in\square_q^\times$, then 
\[\sqrt{\Delta(P)}^q=\sqrt{\Delta(P)},\]
and so 
\[z^{q+1}=\frac{H(P)^2+\Delta(P)-2H(P)\sqrt{\Delta(P)}}{4Q^{q+1}}.\] Then $z^{q+1}=1$ if and only if $\Delta(P)-H(P)\sqrt{\Delta(P)}=0$. But $\Delta(P)\ne 0$, so $H(P)=\sqrt{\Delta(P)}$, which implies $H(P)^2=\Delta(P)$ and hence $Q(P)=0$, which is against our assumption. Hence the tensor is nonsingular. Summing up, we can state the following result.
\begin{proposition}\label{nonsingularity}
The point $P\in PG(3,q^2)$ is nonsingular if and only if the evaluation of $\Delta$ at $P$ is a non-zero square in $\mathbb{F}_q$, namely if and only if $\Delta(P)\in\square^\times_q$. 
\end{proposition}

The matrices associated to $\cQ$ and $\cH$ commute, therefore \textit{ $\cH$ and $\cQ$ induce commuting polarities $\perp_\cH$ and $\perp_\cQ$, and the fixed points of the collineation $\rho=\perp_\cH\perp_\cQ$ given by the product of the two polarities form a subgeometry $\mathrm{Fix}(\rho)\simeq \PG(3,q)$}.
Since $\rho$ acts as $\rho(\alpha,\beta,\gamma,\delta)=(\delta^q,\gamma^q,\beta^q,\alpha^q)$, the subgeometry fixed by $\rho$ is the one corresponding to the Dickson matrices, namely 
\[
\mathrm{Fix}(\rho)=\Sigma.
\]
Commuting polarities has been studied first by Segre in \cite{Segre1965}, and investigated later by several authors in the context of finite geometry, and they exhibit a number of interesting behaviours. From now on, we denote by $\perp$ the correlation $\perp_\cH$. 
Denote by $\cQ_0$ the intersection of $\cQ$ and the subgeometry $\Sigma$. Then $\cQ_0$ is also the intersection of $\cH$ and $\Sigma$ and is itself a quadric in the subgeometry $\Sigma$.
Since $\cQ_0$ fully contains sublines of $\Sigma$, it is a hyperbolic quadric of $\Sigma$.
We refer to a line meeting $\Sigma$ in a subline as a $\Sigma$-line. Similarly, we refer to a $\Sigma$-line meeting $\cQ_0$ in $1,2$ or $q+1$ points as $\cQ_0$-tangent, $\cQ_0$-secant or $\cQ_0$-generator respectively and a plane meeting $\Sigma$ in a plane which is tangent to $\cQ_0$ as $\cQ_0$-tangent plane. Note that a line of $\Sigma^*$ meeting $\cQ_0$ in $1$ point is not necessarily a $\Sigma$-line.
Observe that since the polarities are commuting, the tangent planes of $\cQ$ and $\cH$ at a point $P$ in their intersection coincide.

The following result is well known and it is a consequence of an easy counting argument.
\begin{result}\label{points are on extended sigma line}
Given a subgeometry $\Sigma\simeq PG(3,q)$ of $PG(3,q^2)$, any point $P$ of $PG(3,q^2)$ is either in $\Sigma$ or belongs to precisely one  $\Sigma$-line.
\end{result}
By Theorem \ref{contraction is subline}, this unique line is the contraction of the point.
A good use of Result \ref{points are on extended sigma line} is to characterize points on particular sets in terms of the type of extended $\Sigma$-lines containing them. A first example is provided for the surface $\cD$ associated with the polynomial $\Delta$.

\begin{proposition}\label{delta properties}
The variety $\cD$ consists of the points of $PG(3,q^2)$ contained on an extended generator or an extended tangent line of $\cQ_0$. Equivalently, $\cD$ consists of points of $\cQ^+\cap\cH$, points of $\Sigma$ and points lying on precisely one $\cQ_0$-tangent plane. In particular, $\cD$ is an instance of the quasi-Hermitian surfaces introduced in \cite{Pavese2015}.
\end{proposition}

\begin{proof}
From the analysis of the System \eqref{system}, a point $P$ belong to $\cD$ if and only if its contraction meets $\cQ_0$ in either one or $q+1$ points. Rephrasing, the contraction is either a $\cQ_0$-tangent or a generator of $\cQ_0$. 
It remains to prove that a point $P\notin\Sigma$ on an extended $\cQ_0$-tangent line lies on precisely one $\cQ_0$-tangent plane. If $P$ belongs to $Q^\perp$ and $R^\perp$, with $Q,R\in\cQ_0$, then the point $P$ is on the line $Q^\perp\cap R^\perp$, which is a $2$-secant of $\cQ_0$. This is a contradiction with Result \ref{points are on extended sigma line}.
\end{proof}

As a consequence, $\Delta=0$ is an equation for the quasi-Hermitian variety constructed in \cite{Pavese2015}, in dimension three. As we will explain in Section \ref{section delta equation}, the quasi-Hermitian varieties constructed in \cite{Pavese2015} always admit an equation of this form.

\begin{proposition} \label{intersection quadric and hermitian }
The intersection of $\cQ$ and $\cH$ consists of $\cQ_0$ and the points contained in $\cQ_0$-generators.
\end{proposition}
\begin{proof}
A $\cQ_0$-generator $\ell$ is a $\Sigma$-line meeting $\cQ_0$ in $q+1$ points. Therefore, $\ell$ must be a generator of $\cQ$ as well, implying $\ell\subseteq P^{\perp_\cQ}\cap\cQ$ for some $P\in\cQ$. Since the polarities commute this means that $\ell$ is a generator of $\cH$ as well, henceforth $\ell\subseteq \cQ\cap\cH$. 
\end{proof}

Let $\cT$ denote the set of points on $\cQ_0$-generators but not on $\cQ_0$.
By Proposition \ref{intersection quadric and hermitian }, we can write 
\[\cQ^+\cap\cH=\cQ_0\sqcup\cT;\]
where $\sqcup$ is the disjoint union symbol. 
We will also use the notation $\Tilde{D}$ to denote the set of points on $\cQ_0$-tangent lines but not in $\Sigma$, and $\Tilde{\Sigma}$ to denote points in $\Sigma$ but not in $\cQ_0$. With this notation we can write 
\[\cD=\cQ_0\sqcup\cT\sqcup\tilde{\Sigma}\sqcup\Tilde{\cD}.\]

The following straightforward result characterises the $\Sigma$-lines defined by points not on $\cQ_0$-tangent planes.

\begin{proposition}
For a point $P\in PG(3,q^2)\setminus\Sigma$, the following two properties are equivalent:
\begin{itemize}
\item $P\in\ell$ with $\ell$ a $\Sigma$-line external to $\cQ_0$;
\item $P$ is not on a $\cQ_0$-tangent plane.
\end{itemize}
\end{proposition}
\begin{proof}
Since a $\Sigma$-line $\ell$ external to $\cQ_0$ is in particular a $\Sigma$-line, it must meet any $\cQ_0$-tangent plane, which is a $\Sigma$-plane, in a point of $\Sigma$. As a consequence, $\ell$ doesn't share any point outside $\Sigma$ with $\cQ_0$-tangent planes. The result follows from Result \ref{points are on extended sigma line}, the fact that $\cQ_0$-tangent lines, $\cQ_0$-secants and $\Sigma$-lines external to $\cQ_0$ induce a partition on the set of lines of $\Sigma$, and $\cQ_0$-tangent lines and $\cQ_0$-secants are contained in $\cQ_0$-tangent planes.
\end{proof}
Putting together Proposition \ref{nonsingularity} and Result \ref{points are on extended sigma line}, we obtain the following geometric characterization of nonsingularity.
\begin{lemma}\label{nonsingular iff external or no q0 tg plane}
A point $P\notin\Sigma$ is nonsingular if and only if it lies on an $\Sigma$-line external to $\cQ_0$.  Equivalently, $P\notin\Sigma$ is nonsingular if and only if it does not lie on a $\cQ_0$-tangent plane.
\end{lemma}
We end this section recalling some well known properties of the group of collineations of $PG(3,q^2)$ stabilizing $\cQ_0$. This group is isomorphic to $PCGO^+(4,q)$, but we point out that, since $\Sigma$ is not the canonical subgeometry, the matrix representation of $PCGO^+(4,q)$ must be modified accordingly. 
\begin{result}\cite[Theorem 15.3.12]{FPSO3D}\label{group properties}
The group $G_0\simeq PCGO^+(4,q)\geq PSO^+(4,q)$ stabilizing the subquadric $\cQ_0$ acts transitively on points of $\cQ_0$ and on points of $\Sigma\setminus \cQ_0$. Moreover, $G_0$ acts transitively on the $\Sigma$-lines tangents, external or secant to $\cQ_0$.
\end{result}

\section{The Geometry of Nonsingular Fourfold Tensors}\label{sec:fourfold tensors}

A fourfold tensor in the tensor product of four copies of $\Fqn$ is nonsingular if and only if its contraction space forms an $n$-dimensional vector subspace of nonsingular threefold tensors. Due to our new understanding of threefold tensors and their contractions, we can study nonsingular fourfold tensors via geometric considerations of subspaces in $\PG(n^2-1,q^n)$.

\begin{lemma}
    The contraction of a nonsingular $2\times 2\times 2\times 2$ tensor defines either a nonsingular point of $\PG(3,q^2)$, or a subline of $\PG(3,q^2)$ consisting of nonsingular points.
\end{lemma}

Clearly in the first case, all such points are equivalent. In the second case, we can subdivide the possibilities further.

\begin{lemma}
Suppose $U$ is a nonsingular $2\times 2\times 2\times 2$ tensor whose contraction defines a subline $\langle S,T\rangle_{\Fq}$ of $\PG(3,q^2)$ consisting of nonsingular points. Then the set of contraction spaces $C_3(U) ;= \{C_3(\lambda S+\mu T):\lambda,\mu\in \Fq\}$ is a set of sublines of $\Sigma$ disjoint from $\cQ_0$. Moreover, the dimension $d(U) := \dim\langle C_3(U)\rangle_{\Fq}$ is invariant under $\GL^4$-equivalence.
\end{lemma}

In \cite{Coolsaet2013}, \cite{Coolsaet2}, this number $d$ was called the {\it $12$-rank}. In those papers, the case where $d=2,3$ were considered. In these articles the standard model was used, where connections with disjoint hyperbolic quadrics were observed arising from an algebraic analysis. Using our cyclic model, we obtain this and further results from a geometric analysis.

\begin{lemma}
    Suppose $U$ is as in the previous lemma, and suppose $d(U)=4$. Then the set of contraction spaces $C_3(U)$ is a regulus in a hyperbolic quadric disjoint from $\cQ_0$, and the subline $\langle S,T\rangle$ is a subline of a line in its opposite regulus of the extension of this quadric to $\Sigma^*$. Conversely, if $\cQ'$ is any hyperbolic quadric of $\Sigma$ disjoint from $\cQ_0$, then any line in the extension of this quadric to $\Sigma^*$ meets the extension of one of the reguli in this quadric in a subline which is the contraction of a nonsingular fourfold tensor.
\end{lemma}

This allows us to fully classify all nonsingular fourfold tensors by classifying such quadrics disjoint from $\cQ_0$.

\begin{theorem}
The equivalence classes of nonsingular fourfold tensors $U$ such that $d(U)=4$ are in one-to-one correspondence with the equivalence classes under the stabiliser of $\cQ_0$ of hyperbolic quadrics disjoint from $\cQ_0$ containing a fixed line disjoint from $\cQ_0$.
\end{theorem}

    

\section{Quasi-Hermitian surfaces}\label{sec:quasi hermitian}

In this section we consider orbits of points of $\PG(3,q^2)$ under the action of the group $G_0$. This analysis naturally gives rise to certain surfaces which we will denote by $S_\xi$; these surfaces turn out to form building blocks for new {\it quasi-Hermitian surfaces}.

A quasi-polar space is a set of points in $\PG(m, q)$, $m \geq2$, with the same intersection sizes with hyperplanes as a non-degenerate classical polar space embedded in $\PG(m, q)$. 
A trivial example of quasi-polar space in $\PG(2n-1,q)$ is given by the set of $q^{n-1}$ elements out of an $(n-1)$-spread, to obtain a quasi-polar space of hyperbolic type and of $q^{(n+1)/2}$ elements out of an $(n-1)$-spread, to obtain a quasi-polar space of hermitian type. 

Quasi-Hermitian varieties in $\PG(2,q^2)$ are classical objects in finite geometry, known as unitals.
For $r\geq 3$ the size of a polar space is characterized by its intersection numbers with respect to hyperplanes, see \cite{SchillewaertVoorde2022}.
The first nontrivial constructions were provided by De Clerck, Hamilton, O'Keefe, and Penttila, in \cite{DHOP} for quadratic type, and by De Winter and Schillewaert in \cite{DS} for hermitian type. Since quasi-polar spaces have two intersection numbers with respect to hyperplanes of $\PG(r, q^2)$, they produce strongly regular graphs and two-weight codes, see \cite{CK}. 
A strongly regular graph $\Gamma(\cH)$, is obtained from the two character set $\cH\subseteq\PG(r,q)$ by embedding $\PG(r, q)$ as a hyperplane $\Pi$ of $\PG(r+1, q)$, and by considering the graph whose vertices are the points of $\PG(r+1, q) \setminus \Pi$ and two vertices are adjacent if they span a line meeting $\cH$. 
Moreover, by \cite[Theorem 3.5]{CRV}, under mild hypotheses, given two quasi-polar spaces $\cH$ and $\cH'$, the graphs $\Gamma(\cH)$ and $\Gamma(\cH')$ are isomorphic if and only if there is a collineation of $\PG(r+1, q)$ fixing $\Pi$ and mapping $\cH$ to $\cH'$. 
 Since the parameters of the strongly regular graph arising in such way from a two character set are determined by the characters, nonisomorphic quasi-polar spaces can produce many nonisomorphic strongly regular graphs with the same parameters. In particular, the construction provided in this paper provides many nonisomorphic strongly regular graphs with the same parameters and the same non trivial automorphism group isomorphic to $\mathrm{PCSO}^+(4,q)$.
\subsection{The surfaces $S_\xi$}
Since the surface $\cD$ consists of the points on extended generators and on extended tangent of $\cQ_0$ it is the quasi-Hermitian variety of \cite{Pavese2015}. As we will prove in Theorem \ref{quasi hermitian construction}, we can see this construction as a degenerate case of a more general construction, which goes hand to hand with the one in \cite{LavrauwLiaPavese2023}, and will lead us to new examples of quasi-Hermitian surfaces. The construction requires the introduction of some surfaces, which are obtained by a deformation of the quadric and the Hermitian surfaces arising from System \eqref{system}.

We define $S_\xi$ to be the surface of equation \[S_\xi:\quad H-2\xi Q^{(q+1)/2}=0.\] 
The set $\cS=\{S_\xi:\xi \in\mathbb{F}_{q^2}\}$ consists of a pencil of surfaces, with $\cQ_0\sqcup \cT$ as base. 
Every point of $PG(3,q^2)$ is contained in one of these surfaces, and, over $PG(3,q^2)$, only $2(q-1)+1$ members of $\cS$ contain points other than the base.
Indeed, since $H^2(P)$ and $Q(P)^{q+1}$ belong to $\fq$ for any $P\in PG(3,q^2)$, for any $\xi$ such that $\xi^2\not\in\fq$, the set $S_\xi(\fq^2)$ is precisely the set of $\mathbb{F}_{q^2}$-rational points of $\cH\cap\cQ$, which is easily seen to be the base of the pencil $\cS$.
The $2(q-1)$ surfaces with size larger than $(q+1)(2q^2-q+1)=\mid\cQ\cap \cH\mid$, are those in the set $\{S_\xi:\xi^{2(q-1)}=1\}$. 
From now on, we denote this set by $\cS$.
The elements of $\cS$ belong to one of two families of surfaces, $\cS^1$ and $\cS^2$ defined below.

\begin{lemma}\label{lemma:tec}\cite[Lemma 2.3]{LavrauwLiaPavese2023}
Define the sets
\begin{align*}
& \cZ_1 = \left\{\xi \in \F_{q^2}\setminus \{\pm 1\}: 1-1/\xi^2\in\square^\times_q\right\}; \\ 
& \cZ_2 = \left\{\xi\in \F_{q^2}: 1-1/\xi^2\not\in\square^\times_q\right\}. 
\end{align*}
Then $\cZ=\cZ_1 \sqcup \cZ_2=\{\xi\in\F_{q^2}:\xi^{2(q-1)}=1\}$, $|\cZ_1|=q-3$, $|\cZ_2|=q-1$ and so $|\cZ_1 \cup \cZ_2| = 2(q-1)$.
\end{lemma}
\begin{proof}
The set $\cZ_1$ can be rewritten as $\cZ_1=\left\{\xi \in \F_{q^2}: \xi^2=1/(1-s^2), s\in\F_q^\times\setminus\{\pm 1\}\right\}$. Since for any value of $s\in\F_q^\times\setminus\{\pm 1\}$ it is possible to find two values of $\xi$ with $s^2+\xi^2=1$, the claim on the size of $\cZ_1$ follows.
The proof for $\cZ_2$ is analogous.
\end{proof}
We define $\cS^1$ and $\cS^2$ to be the following sets:
\[\cS^1:=\{S_\xi:\xi \in \cZ_1\},\quad
\cS^2:=\{S_\xi:\xi \in \cZ_2\}.
\]
We also define, for ease of reference in the remainder of the section, the set $\Tilde{\cS}$ as $\Tilde{\cS}=\{S_\xi\setminus\cH\mid \xi\in\cZ\}$, and we define $\Tilde{\cS}^1$ and $\Tilde{\cS}^2$ in an analogous way.
A first observation is that the surface $\cD$ consists precisely of the join of the surfaces $S_1$ and $S_{-1}$.

\begin{lemma}
All the points of $\tS$, with $\tS$ in $\tcS^1$ are nonsingular, and all the points of $\tS$, with $\tS$ in $\tcS^2$ are singular.
\end{lemma}
\begin{proof}
If $P\in S_\xi$, then $H(P)-2\xi Q(P)^{(q+1)/2}=0$, therefore $\Delta(P)=H(P)^2(1-1/\xi^2)$. The claim follows.
\end{proof}

Although the group $G$, stabilizing $\cQ_0$, does not stabilize each of the sets $\tS$, it possesses a subgroup of index $2$ which does stabilise each $\tS$. Moreover, $G$ stabilises the union of certain pairs of these sets.

    Let $G_0$ denote the subgroup of $G$ induced by elements of the form $A\otimes B$, where $A$ and $B$ are Dickson matrices satisfying 
    \[
    \det(A)\det(B)\in (\Fq^\times)^2.
    \]
    This group is isomorphic to $\mathrm{PCSO}^+(4,q)$.

\begin{proposition}\label{surfacesstabilized}
The group $G$ preserves each union of the type $\tS_\xi\cup\tS_{-\xi}$, with $\tS_\xi\in\tcS$. The group $G_0$ preserves each set in $\tcS$. Moreover, $G_0$ preserves the sets $\cH_1$ and $\cH_2$ of nonsingular and singular points of $\cH\setminus\cQ$.
\end{proposition}
\begin{proof}
Let $A=\begin{pmatrix}
a & b   \\
b^q & a^q \\
\end{pmatrix}$ and $B=\begin{pmatrix}
c & d   \\
d^q & c^q \\
\end{pmatrix}$ be nonsingular matrices. Then we can consider $C=A\otimes B$ to be an element of $G$, and every element of $C$ is induced by a matrix of this form. 
Then $C$ preserves $\tS_\xi\cup\tS_{-\xi}$ if and only if 
\[
H(P)^2-4\xi^2Q(P)^{q+1}=0
\]
implies 
\[
(C^TH(P)C^q)^2-4\xi^2(C^TQ(P)C)^{q+1}=0.
\]
Since 
\[
C^THC^q=\det(A)\det(B)Q; \quad C^TQC=\det(A)\det(B)Q; 
\]
the first claim is true.
Moreover, $C$ preserves $\tS_\xi$ if and only if 
\[
H(P)-2\xi Q(P)^{(q+1)/2}=0
\]
implies 
\[
C^TH(P)C^q-2\xi(C^TQ(P)C)^{(q+1)/2}=0.
\]
This occurs if and only if $\det(A)\det(B)$ is a nonzero square in $\Fq$, and hence the second claim is true as well.
The last statement is a consequence of the fact that $G$ preserves $\cH$ and preserves the singularity of points.
\end{proof}
\begin{proposition} \label{sigmalines external to q0}
A $\Sigma$-line external to $\cQ_0$ meets each of the surfaces $\tS$, with $\tS\in\tcS^1$, in $q+1$ (nonsingular) points.
\end{proposition}
\begin{proof}
By the transitivity of $G$ we can choose $\ell=P_1P_2$, with $P_1=(0,1,1,0)$ and $P_2=(0,1,-1,0)$. Points of $\ell$ are $P_1$ or of the form $(0,\lambda+1,\lambda-1,0)$, with $\lambda\in\F_{q^2}$. Such a point belongs to $\cS_\xi$ if and only if 
\[
2(\lambda^{q+1}+1)+2\xi(1-\lambda^2)^{(q+1)/2}=0.
\]
Since the equation above has degree $q+1$, there are at most $q+1$ points on each $\cS_\xi$. Since $\ell$ has precisely $q+1$ points on $\Sigma$, $2$ points on $\cQ$ and $q+1$ points on $\cH$, and can not have any point on a set contained in $\tcS^2$, the claim follows.
\end{proof}
\begin{corollary}
For any $S$ in $\cS^1$, the size of $S$ is $(q-1)^2q^2(q+1)/2$.
\end{corollary}
\begin{proof}
By Propositions \ref{surfacesstabilized} and \ref{sigmalines external to q0}, the size of $S$ is
\[
\mid (\cQ\cap \cH)\mid + \mid\{\Sigma\textit{-lines external to }\cQ_0\}\mid(q+1).
\]
Since the number of $\Sigma$-lines external to $\cQ_0$ is $q^2(q-1)^2/2$, the claim follows.
\end{proof}

\begin{lemma}\label{hermitian cone preparatory lemma}
Let $P\in\cQ_0$ and $R\in S\cap P^\perp$, with $S\in\cS$. Then the line $\ell=PR$ lies on $S$. 
\end{lemma}
\begin{proof}
By the transitivity of $G$, we can assume $P$ to be the point $(1,1,1,1)$. Let $R=(x,y,z,t)$ be a point in $S_\xi\cap P^\perp$. Then $x-y-z+t=0$.
The point $P_\lambda=P+\lambda R$ of $\ell$ belongs to $S_\xi$ if and only if 
$H(P_\lambda)-2\xi Q(P_\lambda)^{(q+1)/2}=0$.
Since $H (P_\lambda)=\lambda^{q+1}H (R)+\lambda^q(x-y-z+t)+\lambda(x-y-z+t)^q$ and $Q(P_\lambda)=\lambda^2Q(R)+\lambda(x-y-z+t)$, the claim follows.
\end{proof}

\begin{proposition}\label{hermitian cone}
Let $P$ be a point of $\cQ_0$ and denote by $\cL_P$ the pencil of lines through $P$ contained in the $\cQ_0$-tangent plane $P^\perp$. Then $\cL_P$ contains:
\begin{itemize}
\item $2$ $\cQ_0$-generators;
\item $q-1$ $\cQ_0$-tangent $\Sigma$-lines;
\item $q-1$ lines contained in $\cH$; 
\item $q-1$ lines contained in the surface $S$, for each $S$ in $\cS^2$.
\end{itemize}
Moreover, all the points different from $P$ of the $q-1$ lines contained in $\cH$ are in $\cH_2$ and all the points different from $P$ of the $q-1$ lines contained in $S$ are in $\tS$. 
\end{proposition}
\begin{proof}
The first two points are consequence of $\cQ_0$ being an hyperbolic quadric.
Since $P$ belongs to $\cH$ as well, the Hermitian surface $\cH$ meets $P^\perp$ in an Hermitian cone, consisting of $q+1$ lines. Since two of these lines are $\cQ_0$-generators, the third claim follows as well.

By Lemma \ref{hermitian cone preparatory lemma}, it is enough to show that $P^\perp$ contains $q-1$ points of $S$. To see this, consider a second point $R\in\cQ_0$, $R\neq P$. 
The line $P^\perp \cap R^\perp$ has $q-1$ points in $\cH_2$, two points in $\cQ_0$, no points in $\hat{S}$ for $\hat{S}\in\tcS^1$, no point in $\cH_1$ and no point in $\Tilde{\cD}$. Since each of the surfaces $S$ has degree $q+1$, it has at most $q+1$ points of $P^\perp \cap R^\perp$, and since two are in $\cQ_0$, at most $q-1$ further points in $\tS$. Since the line must have $q^2+1$ points, which is obtained by taking the maximum possible intersections, the claim follows. 
\end{proof}
As a consequence, we obtain an analogous of Proposition \ref{sigmalines external to q0} for $\Sigma$-lines secant to $\cQ_0$ and points of $\tS$ with $\tS\in\tcS^2$.
\begin{proposition}\label{secant lines meet S2}
A $\Sigma$-lines secant to $\cQ_0$ contains $q+1$ points of $\tS$, for any $\tS$ in $\tcS^2$.
\end{proposition}

We can characterize the various sets we defined so far in terms of the number of $\cQ_0$-tangent planes passing through each point. 
\begin{lemma}\label{number of tg planes per point}
For a point $P\in PG(3,q^2)$, the following holds true:
\begin{itemize}
\item if $P$ belongs to $\cQ_0$, then $P$ is contained in precisely $2q+1$ $\cQ_0$-tangent planes;

\item if $P$ belongs to $\Tilde{\Sigma}$ or $\cT$ then $P$ is contained in precisely $q+1$ $\cQ_0$-tangent planes;
\item if $P$ belongs to $\Tilde{\cQ}$, any $S$ with $S\in\tcS^1$ or $\cH_1$, then $P$ is contained in no $\cQ_0$-tangent planes;
\item if $P$ belongs to $\Tilde{S}$ with $\Tilde{S}\in\tcS^2$ or $\cH_2$, then $P$ is contained in precisely two $\cQ_0$-tangent planes;
\item if $P$ belongs to $\Tilde{\cD}$ then $P$ is contained in precisely one $\cQ_0$-tangent plane.
\end{itemize}
\end{lemma}
\begin{proof}
By duality, if $P$ belongs to $\cQ_0$ or $\Tilde{\Sigma}$, then $P^\perp$ meets $\cQ_0$ in either a degenerate or a non-degenerate conic, explaining the first two points.
If $P$ belongs to $\cT$, then it is on $Q^\perp$ for any $Q$ in the (unique) extended generator in which $P$ lies. Observe that, since $P^\perp$ meets $\Sigma$ precisely one sub-line, these are all.
By Lemma \ref{nonsingular iff external or no q0 tg plane}, points on $\cH_1$ and on $S$, with $S\in\tcS^1$ are on no $\cQ_0$-tangent planes.
Finally, let $P$ be on at least three $\cQ_0$-tangent planes. Then $P^\perp$ has at least three points of $\Sigma$. If they are on a line, that must be a $\cQ_0$-generator, otherwise, they are on a conic, in both case $P$ will belong to the tangent plane to each of those $q+1$ points.
By Proposition \ref{hermitian cone}, each point of $\cH_2$ and any $\tS$ in $\tcS^2$ lies on a $\cQ_0$-secant, which is unique by Result \ref{points are on extended sigma line}. As a consequence, since $\cQ_0$-secants lay on precisely two $\cQ_0$-tangent planes, the claim follows.
\end{proof}

\begin{proposition}
Each of the sets $\tS$, with $\tS\in\tcS$ is an orbit under the action of $G_0$. The sets $\cH_1$ and $\cH_2$ are orbits under the action of $G_0$.
\end{proposition}
\begin{proof}
First, suppose $\tS\in\tcS^2$.
Since each point of $\tS$ lies on precisely one $\cQ_0$-secant $\ell$ and the group $G$ is transitive on such lines, it is enough to show that the stabilizer $G_\ell$ of $\ell$ is transitive on the points of $\ell\cap\tS$. We will show that this is true for a subgroup $H$ of $G_\ell$. The group $G$ has a normal subgroup $\hat{G}$ of index two isomorphic to $PGL(2,q)\times PGL(2,q)$. Let $H$ be $G_\ell\cap\hat{G}$. Without loss of generality, we can fix $\ell$ to be the line $P_1P_2$ with $P1=e_1\otimes f_1$ and $P_2=e_2\otimes f_2$. Any $g\in H$ is of the form $g=g_1\times g_2$, acting on $v\otimes w$ as $(v\otimes w)^g=v^{g_1}\otimes w^{g_2}$. Observe that the group we are dealing with is not directly a subgroup of $PGL(2,q)\times PGL(2,q)$ in its canonical representation, but, since we only care about the orbit structure, we will use the canonical representation.
By the hypothesis $g\in H$, it follows $P_i^g=P_i$ or $P_i^g=P_j$, with $i\neq j$ and $i,j\in\{1,2\}$. This implies that $g_1$ acts on $e_1,e_2$ as $e_1^{g_1}=\lambda_1e_1$ and $e_2^{g_1}=\lambda_2e_2$ or as $e_1^{g_1}=\mu_1e_2$ and $e_2^{g_1}=\mu_2e_1$ and $g_2$ acts on $f_1,f_2$ fixing both. In particular, $H$ has size $2(q-1)^2$. Remember that, since $H$ fixes $\ell$, it must induce a subgroup of $PGL(2,q)$. The induced action could have non-trivial kernel. Any transformation of $H$ acting non trivially on $\ell$ will have either $0$, $1$ or $2$ fixed points. Since only elements of order $p$ fix precisely one point and elements of order dividing $q^2-1$ have two fixed points, all the elements with non-trivial action, have two fixed points.
Observe that $g\in H$ fixes all the points $P_\lambda\in\ell$ if and only if acts of on $e_1,e_2$ as $e_1^{g_1}=\lambda_1e_1$ and $e_2^{g_1}=\lambda_2e_2$ with $\lambda_1=\lambda_2$. Therefore, there are $q-1$ such elements. By the lemma which is not Burnside's, the number of orbits of $H$ on points of $\ell$ must be
\[
(2(q-1)(2q-3)+(q^2+1)(q-1) )/ (2(q-1)^2)=(q+5)/2.
\]
One of these orbits is $\{P_1,P_2\}$, one is contained in $\Tilde{\Sigma}\cap\ell$, one is contained in $\cH_2\cap\ell$ and, for each $\xi\in\cZ_2$, one is contained $\tS_\xi\cap\ell$.
Since the latter are at least $q-1$ orbits, this implies that each of the sets $\tS_\xi\cap\ell$ is indeed an orbit.

Suppose now $\tS\in\tcS^1$. We will apply the same method, but using a $\Sigma$-line $\ell$ external to $\cQ_0$. The line $\ell$ will therefore be a secant for the quadric $\cQ$, namely $\ell=P_1P_2$, with $P_i\in\cQ\setminus\cQ_0$. Without loss of generality, we can again assume $P_1=e_1\otimes f_1$ and $P_2=e_2\otimes f_2$. With this choice though, the group representation we have to use is not the canonical one. Here $g\in H$ acts on $e_1\otimes f_1$ as $(e_1\otimes f_1)^g=\lambda e_1\otimes \mu f_1$ and on $e_2\otimes f_2$ as $(e_2\otimes f_2)^g=\lambda^q e_2\otimes \mu^q f_2$, or on $e_1\otimes f_1$ as $(e_1\otimes f_1)^g=\lambda e_2\otimes \mu f_2$ and on $e_2\otimes f_2$ as $(e_2\otimes f_2)^g=\lambda^q e_1\otimes \mu^q f_1$. This amounts to a total of $2(q+1)^2$ elements, namely $|H|=2(q+1)^2$. 
Reasoning as before, we obtain that there are as many as
\[
(2(q+1)(2q+1)+(q^2+1)(q+1) )/ (2(q+1)^2)=(q+3)/2
\]
orbits.
\end{proof}

\begin{theorem}\label{pointorbits}
The group $G\leq PGL(4,q^2)$ stabilizing the subquadric $\cQ_0$ acts on the points of $PG(3,q^2)$ with the following orbits:
\begin{itemize}
\item $\cQ_0$, of size $(q+1)^2$;
\item $\Tilde{\Sigma}$, of size $q^3-q$;
\item $\cT$, of size $2(q+1)(q^2-q)$;
\item $\Tilde{\cQ}$, of size $q^4-2q^3+q^2$;
\item $\Tilde{\cD}$ of size $q(q^2-1)^2$;
\item the set $\cH_1$ of nonsingular points of $\cH$, of size $(q+1)(q-1)^2q^2/2$;
\item the set $\cH_2$ of singular points of $\cH$ not in $\cQ_0$, of size $(q-1)q^2(q+1)^2/2$;
\item the $(q-3)/2$ orbits in $\tcS^1$, each of size $(q+1)(q-1)^2q^2$ and given by $\tS_\xi\cup\tS_{-\xi}$, with $\xi\in\cZ_1$;
\item the $(q-1)/2$ orbits in $\tcS^2$, each of size $(q-1)q^2(q+1)^2$ and given by $\tS_\xi\cup\tS_{-\xi}$, with $\xi\in\cZ_2$.
\end{itemize}
\end{theorem}
\begin{remark}
By Proposition \ref{surfacesstabilized}, the group $G_0$ splits the orbit given by pairs of sets in $\tcS$ in the two sets themselves. Also, the set $\Tilde{\Sigma}$ splits in two orbits.
\end{remark}

\subsection{Point orbit distribution of planes}
\label{section plane intersection numbers}
In this section we determine how planes of $\PG(3,q^2)$ intersect certain point sets related to the $G$- and $G_0$-orbits. This will allow us to determine new quasi-Hermitian varieties as joins of some of these sets. We proceed by considering the planes $P^\perp$ for points $P$ representing each orbit.

\begin{proposition}\label{Q0tgplane pts distribution}
Let $P\in\cQ_0$. The points of $P^\perp$ are distributed as in following table.
\\
\begin{tabular}{|c||c|c|c|c|c|c|c|c|c|}
\hline
   $\cO$: & $\cQ_0$ & $\Tilde{\Sigma}$ & $\cT$ & $\Tilde{\cQ}$ & $\Tilde{\cD}$ & $\cH_1$ & $\cH_2$ & $\tS\in\tcS^1 $ & $\tS\in\tcS^2$  \\ 
   \hline
$|P^\perp \cap \cO|$: & $2q+1$ & $q^2-q$ & $2(q^2-q)$ & $0$ & $(q-1)(q^2-q)$ & $0$ & $q^3-q^2$ & $0 $ & $(q-1)q^2$  \\ 
\hline
\end{tabular}
\end{proposition}
\begin{proof}
Since $P\in\cQ_0\cap\cH$, $P^\perp$ is a tangent plane for both $\cQ_0$ and $\cH$, and the two surfaces share two generators (meeting $\cT$ in $q^2-q$ points each). The remaining $q-1$ $\Sigma$-lines are tangent to $\cQ_0$. Therefore they have $q$ points on $\Tilde{\Sigma}$ and $q^2-q$ on $\Tilde{\cD}$. 
Nonsingular points are those not contained in any $\cQ_0$-tangent plane. 
Since $P$ belongs to $\cH$, $P^\perp\cap\cH$ is a degenerate Hermitian variety. Since there are no points in $P^\perp\cap\cH_1$, there must be $(q^3+q^2+1)-(2q+1)-2(q^2-q)$ in $P^\perp\cap\cH_2$. Finally Proposition \ref{hermitian cone} explains the case $\cO\in\tcS^2$.
\end{proof}

\begin{proposition}\label{sigmaplane pts distribution}
Let $P\in\Tilde{\Sigma}$. The points of $P^\perp$ are distributed as in following table.
\\
\begin{tabular}{|c||c|c|c|c|c|c|c|c|c|}
\hline
   $\cO$: & $\cQ_0$ & $\Tilde{\Sigma}$ & $\cT$ & $\Tilde{\cQ}$ & $\Tilde{\cD}$ & $\cH_1$  \\ 
   \hline
$|P^\perp \cap \cO|$: & $q+1$ & $q^2$ & $0$ & $q^2-q$ & $(q^2-q)(q+1)$ & $(q^2-q)(q+1)/2$   \\ 
\hline
\end{tabular}

\begin{tabular}{|c||c|c|c|c|c|c|c|c|c|}
\hline
   $\cO$: &  $\cH_2$ & $\tS\in\tcS^1 $ & $\tS\in\tcS^2$  \\ 
   \hline
$|P^\perp \cap \cO|$: & $(q^2-q)(q+1)/2$ & $(q^2-q)(q+1)/2 $ & $(q^2-q)(q+1)/2$  \\ 
\hline
\end{tabular}
\end{proposition}
\begin{proof}
Since $P^\perp$ is a $\Sigma$-plane and $P\not\in\cQ_0$, it meets $\cQ_0$ in a conic. 
No $\cQ_0$-generator lies in $P^\perp$, and therefore no point of $\cT$.
Since there are $(q^2-q)/2$ lines in $PG(2,q)$ external to an irreducible conic, $P^\perp$ contains $(q^2-q)/2$ $\Sigma$-lines external to $\cQ_0$. This, togheter with Proposition \ref{sigmalines external to q0}, explain the cases $\cO\in\{\cH_1,\Tilde{\cQ}\}\cup \tcS^1$. 
Every $\cQ_0$-tangent at a point on $P^\perp$ plane meets $P^\perp$ in precisely $q-1$ points of each $\Tilde{S}\in\tcS^2$, giving $(q-1)(q+1)q$, the same holds for $\cH_2$ and $\Tilde{\cD}$, explaining the remaining cases.
\end{proof}

The case $P\in\tS$, with $\tS\in\tcS^1$, it's a bit different from all the other. In particular, studying the size of the intersection of $P^\perp$ with $\hat{S}$, for $\hat{S}\in\tcS^2$ can not be dealt with the same kind of arguments used so far, and it is a bit longer. Since it will be useful to simplify counts for the other point distributions, we state here the full result, but, to avoid interruptions in the flow of the proofs, we will leave the proof requiring a different machinery to the Appendix.

\begin{proposition}\label{S1 partial pts distribution}
Let $P\in\tS$, with $\tS \in\tcS^1$. The points of $P^\perp$ are distributed as in following table.
\\
\begin{tabular}{|c||c|c|c|c|c|c|c|}
\hline
   $\cO$: & $\cQ_0$ & $\Tilde{\Sigma}$ & $\cT$ & $\Tilde{\cQ}$ & $\Tilde{\cD}$ & $\cH_1$ \\ 
   \hline
$|P^\perp \cap \cO|$: & $0$ & $q+1$ & $2(q+1)$ & $q^2-2q-1$ & $(q-2)(q+1)^2$ & $(q+1)(q^2-2q-1)/2$   \\ 
\hline
\end{tabular}

\begin{tabular}{|c||c|c|c|}
\hline
   $\cO$: &$\cH_2$ & $\tS\in\tcS^1 $ & $\tS\in\tcS^2 $ \\ 
   \hline
 $|P^\perp \cap \cO|$: &$(q+1)(q^2-2q-1)/2+q^2 $ &$(q+1)(q^2-2q-1)/2$& $(q-1)(q+1)^2/2$\\ 
\hline
\end{tabular}
\end{proposition}
\begin{proof}
Since, from Lemma \ref{number of tg planes per point}, points on $\tS$ are on no $\cQ_0$-tangent plane, by duality $P^\perp$ has no point in $\cQ_0$. Since $P^\perp$ must intersect $\Sigma$ in a line, that must be a $\Sigma$-line external to $\cQ_0$, explaining the case $\cO=\Tilde{\Sigma}$.
On the other hand, each of the $q+1$ points of $P^\perp\cap\Tilde{\Sigma}$ lies on precisely $q+1$ $\cQ_0$-tangent planes. Since there are a total of $(q+1)^2$ of such planes, it follows that each of them contains precisely one point of $\Tilde{\Sigma}$. As a consequence, each such plane meets $P^\perp$ in precisely one point of $\Tilde{\Sigma}$. This implies that it meets $\Tilde{\cD}$ in $q-2$ points, explaining the case $\cO=\Tilde{\cD}$.
The cases $\cO\in\{\cH_2,\cT\}\cup \tcS^2$ are obtained by looking at the intersections with the $\cQ_0$-tangent planes, and using Proposition \ref{hermitian cone} and Lemma \ref{number of tg planes per point}. The sizes of $P^\perp\cap \Tilde{\cQ}$ and $P^\perp\cap\cH_1$ are consequences of the intersection properties of the quadrics and Hermitian surfaces.
The cases $\cO\in\tcS^1$ are more involved; the remainder of the proof can be found in the Appendix.
\end{proof}

\begin{proposition}\label{T pts distribution}
Let $P\in\cT$. The points of $P^\perp$ are distributed as in following table.
\\
\begin{tabular}{|c||c|c|c|c|c|c|}
\hline
   $\cO$: & $\cQ_0$ & $\Tilde{\Sigma}$ & $\cT$ & $\Tilde{\cQ}$ & $\Tilde{\cD}$ & $\cH_1$  \\ 
   \hline
$|P^\perp \cap \cO|$: & $q+1$ & $0$ & $q^2$ & $q^2-q$ & $(q^2-q)(q+1)$ & $(q^2-q)(q+1)/2$   \\ 
\hline
\end{tabular}\\

\begin{tabular}{|c||c|c|c|}
\hline
   $\cO$: & $\cH_2$ & $\tS\in\tcS^1 $ & $\tS\in\tcS^2$  \\ 
   \hline
$|P^\perp \cap \cO|$: & $(q^2-q)(q+1)/2$ & $(q^2-q)(q+1)/2 $ & $(q^2-q)(q+1)/2$  \\ 
\hline
\end{tabular}
\end{proposition}
\begin{proof}
Since $P\in\cT$, $P^\perp$ contains a $\cQ_0$-generator and meets $\cQ$ in a further generator, not extension of one of $\cQ_0$. 
If $\ell$ is the unique $\cQ_0$-generator on which $P$ lies, then $\ell$ has $q^2-q$ points of $\cT$. Through $P$ there is a further $\cQ$-generator $\ell'$, not meeting $\cQ_0$. Since each generator of a quadric meets each line of the opposite ruling of the quadric, $\ell'$ meets the remaining $q$ $\cQ_0$-generators of the ruling containing $\ell$ in one point each. Therefore we have a total of $q^2$ points of $\cT$ and $q^2-q$ points of $\Tilde{\cQ}$.
If a further point of $\Sigma$ was on $P^\perp$, then $P^\perp$ would be an extended $\Sigma$-plane, which is not, namely there are no points in $\Tilde{\Sigma}$.
For each of the $q^2+q$ points of $\cQ_0$ not lying on $P^\perp$, the $\cQ_0$ tangent plane at that point meets $P^\perp$ in a line with $q-1$ points of $\Tilde{\Delta}$, of $\cH_2$ and $\Tilde{S}$ for any $\Tilde{S}\in\tcS^2$, see Proposition \ref{hermitian cone}. The claim for these orbits follows from Lemma \ref{number of tg planes per point}.

All the remaining points must be on the orbits $\cH_1$ and $\Tilde{S}\in\tcS^1$.  The case $\cO\in\tcS^1$ can be obtained from Proposition \ref{S1 partial pts distribution} using double counting taking into account the duality, and the case $\cO=\cH_1$ is a consequence of the fact that $\cH$ is a two character set, by looking at the intersection sizes computed so far.
\end{proof}

\begin{proposition}\label{Q pts distribution}
Let $P\in\Tilde{\cQ}$. The points of $P^\perp$ are distributed as in following table.
\\
\begin{tabular}{|c||c|c|c|c|c|c|}
\hline
   $\cO$: & $\cQ_0$ & $\Tilde{\Sigma}$ & $\cT$ & $\Tilde{\cQ}$ & $\Tilde{\cD}$ & $\cH_1$  \\ 
   \hline
$|P^\perp \cap \cO|$: & $0$ & $q+1$ & $2(q+1)$ & $2(q^2-q)-1$ & $(q-1)(q^2-q+1)$ & $(q+1)(q^2-2q-1)/2$  \\ 
\hline
\end{tabular}

\begin{tabular}{|c||c|c|c|}
\hline
   $\cO$:  & $\cH_2$ & $\tS\in\tcS^1 $ & $\tS\in\tcS^2$  \\ 
   \hline
$|P^\perp \cap \cO|$: &  $(q-1)(q+1)^2/2$ & $(q+1)(q^2-2q-1)/2 $ & $(q-1)(q+1)^2/2$  \\ 
\hline
\end{tabular}
\end{proposition}
\begin{proof}
The cases $\cO\in\{\cQ_0,\cT,\Tilde{\cQ}\}$ are immediate consequence of $P^\perp$ being a tangent plane for the quadric $\cQ$.
By counting the couples $(Q,P^\perp)$ with $Q\in\Tilde{\Sigma}$ and $P\in\Tilde{\cQ}$, using Proposition \ref{sigmaplane pts distribution} and the duality, we get the case $\cO=\Tilde{\Sigma}$.
Each point of $\cH_2$ and $\tS$, with $\tS\in\tcS^2$, is on precisely two $\cQ_0$-tangent plane, and each of these planes meet $P^\perp$ in $q-1$ points (this last statement is consequence of Lemma \ref{hermitian cone}). This explains the cases $\cO\in\{\cH_2\}\cup\tcS^2$. As a consequence of the intersection properties of the Hermitian surface, the case $\cO=\cH_1$ follows.
Using again double counting argument on the couples $(Q,P^\perp)$ with $Q\in\tS$, with $\tS\in\tcS^1$ and $P\in\Tilde{\cQ}$, using Proposition \ref{S1 partial pts distribution} and the duality, we get the case $\cO\in\tcS^1$
The remaining points must belong to $\Tilde{\cD}$.
\end{proof}

\begin{proposition}\label{Delta pts distribution}
Let $P\in\Tilde{\cD}$. The points of $P^\perp$ are distributed as in following table.
\\
\begin{tabular}{|c||c|c|c|c|c|c|}
\hline
   $\cO$: & $\cQ_0$ & $\Tilde{\Sigma}$ & $\cT$ & $\Tilde{\cQ}$ & $\Tilde{\cD}$ & $\cH_1$\\ 
   \hline
$|P^\perp \cap \cO|$: & $1$ & $q$ & $2q$ & $q^2-2q$ & $q^3+q^2-q$ & $q((q+1)/2)^2$ \\ 
\hline
\end{tabular}

\begin{tabular}{|c||c|c|c|}
\hline
   $\cO$: &  $\cH_2$ & $\tS\in\tcS^1 $ & $\tS\in\tcS^2$  \\ 
   \hline
$|P^\perp \cap \cO|$:  & $(q^2+2q)(q-1)/2$ & $q((q+1)/2)^2$ & $(q^2+2q)(q-1)/2$  \\ 
\hline
\end{tabular}
\end{proposition}
\begin{proof}
Since each point of $\Tilde{\cD}$ is on precisely one $\Sigma$-plane tangent to $\cQ_0$, by duality, we obtain the values for $\cO=\cQ_0,\Tilde{\Sigma}$.
Out of the $(q+1)^2$ $\Sigma$-planes tangent to $\cQ_0$, all but one meet $P^\perp$ in $q-1$ points of each of $\tS$, with $\tS\in\tcS^2$ and $\cH_2$, and one in a line containing a point of $\cQ_0$. This line is the unique $\Sigma$-line lying on $P^\perp$ and it is tangent to $\cQ_0$. Moreover, all but $q$ of the $q^2+2q$ $\Sigma$-planes tangent to $\cQ_0$ without $\Sigma$-lines meet $\cT$ in two points. Since each point of $\cT$ is contained in $q+1$ such planes, the value $|P^\perp\cap \cT|$ is explained. The case $\cO=\Tilde{\cQ}$ is a consequence of the intersection properties of quadrics.
Out of the $q^2+2q$ $\Sigma$-planes tangent to $\cQ_0$ without $\Sigma$-lines, $q$ meet $\Tilde{\cD}$ in precisely one point, the remaining $q^2+q$ in $q-1$ points. This explains the case $\cO=\Tilde{\cD}$.
The intersection properties of Hermitian surfaces, together with the value computed for $\cO=\cH_2$ explains the case $\cO=\cH_1$.
To get the value $|P^\perp\cap\tS|$, with $\tS\in\tcS^1$, it is enough to double count on the couples $(Q,P^\perp)$ with $Q\in\Tilde{\cD}$ and $P\in\tS$, using Proposition \ref{S1 partial pts distribution} and the duality.
\end{proof}

\begin{proposition}\label{H2 pts distribution}
Let $P\in\cH_2$. The points of $P^\perp$ are distributed as in following table.
\\
\begin{tabular}{|c||c|c|c|c|c|c|}
\hline
   $\cO$: & $\cQ_0$ & $\Tilde{\Sigma}$ & $\cT$ & $\Tilde{\cQ}$ & $\Tilde{\cD}$ & $\cH_1$\\ 
   \hline
$|P^\perp \cap \cO|$: & $2$ & $q-1$ & $2(q-1)$ & $(q-1)^2$ & $(q-1)^2(q+1)$ & $(q+1)(q-1)^2/2$ \\ 
\hline
\end{tabular}

\begin{tabular}{|c||c|c|c|}
\hline
   $\cO$: &  $\cH_2$ & $\tS\in\tcS^1 $ & $\tS\in\tcS^2$  \\ 
   \hline
$|P^\perp \cap \cO|$: &$((q+1)^2-2)(q-1)/2+q^2$ & $(q+1)(q-1)^2/2$ & $((q+1)^2-2)(q-1)/2$  \\ 
\hline
\end{tabular}
\end{proposition}
\begin{proof}
Since any point of $\cH_2$ is on two $\cQ_0$ tangent planes, by duality, $P^\perp$ has two points of $\cQ_0$. 

As a consequence, the plane $P^\perp$ meets $\Sigma$ precisely in a subline, which is a $2$-secant for $\cQ_0$, and therefore in $q-1$ points of $\Tilde{\Sigma}$.
The value for $\cO=\cT$ ($\cO=\Tilde{\cD}$) is obtained by double counting the couples $(Q,P^\perp)$ with $Q\in\cT$ ($Q\in\Tilde{\cD}$) and $P\in\cH_2$, using Proposition \ref{T pts distribution} (Proposition \ref{Delta pts distribution}) and the duality.
Each of the $(q+1)^2-2$ points $P_0$ of $\cQ_0$ not on $P^\perp$ is such that $P_0^\perp$ meets $P^\perp$ in precisely $q-1$ points. For the remaining two points $R,S$, the tangent dual planes $R^\perp$ and $S^\perp$ meet $P^\perp$ in a line $\ell_R$, $\ell_S$, which, since contains both $R$ and $P$ must be a generator. Considering that each point in $\cH_2$ and $\tS\in\tcS^2$ is contained in two tangent planes, we obtain the value for $\cO\in\{\cH_2\}\cup\tcS^2$. The value for $\cO=\cH_1$ follows from this and the intersection properties of the Hermitian surface.  
The value $|P^\perp\cap\tS|$, with $\tS\in\tcS^1$, is obtained by double counting on the couples $(Q,P^\perp)$ with $Q\in\cH_2$ and $P\in\tS$, using Proposition \ref{S1 partial pts distribution} and the duality.

\end{proof}

\begin{proposition}\label{S2 pts distribution}
Let $P\in\tS$, with $\tS\in\tcS^2$. The points of $P^\perp$ are distributed as in following table.
\\
\begin{tabular}{|c||c|c|c|c|c|c|c|}
\hline
   $\cO$: & $\cQ_0$ & $\Tilde{\Sigma}$ & $\cT$ & $\Tilde{\cQ}$ & $\Tilde{\cD}$ & $\cH_1$& $\cH_2$   \\ 
   \hline
$|P^\perp \cap \cO|$: & $2$ & $q-1$ & $2(q-1)$ & $(q-1)^2$ & $(q-1)^2(q+1)$ & $(q-1)^2(q+1)/2$ & \small{$((q+1)^2-2)(q-1)/2$}\\ 
\hline
\end{tabular}
\begin{tabular}{|c||c|c|c|}
\hline
   $\cO$: & $\tS\in\tcS^1 $ & $\tS\neq\hat{S}\in\tcS^2$&$\tS$  \\ 
   \hline
$|P^\perp \cap \cO|$:  & $(q-1)^2(q+1)/2$ & $((q+1)^2-2)(q-1)/2$& $((q+1)^2-2)(q-1)/2+q^2$  \\ 
\hline
\end{tabular}
\end{proposition}
\begin{proof}

The proof is completely analogous to that of Proposition \ref{H2 pts distribution}, replacing $\tS$ for $\cH_2$ when necessary.
\end{proof}

\begin{proposition}\label{H1 pts distribution}
Let $P\in\cH_1$. The points of $P^\perp$ are distributed as in following table.
\\
\begin{tabular}{|c||c|c|c|c|c|c|c|}
\hline
   $\cO$: & $\cQ_0$ & $\Tilde{\Sigma}$ & $\cT$ & $\Tilde{\cQ}$ & $\Tilde{\cD}$ & $\cH_1$ \\ 
   \hline
$|P^\perp \cap \cO|$: & $0$ & $q+1$ & $2(q+1)$ & $q^2-2q-1$ & $(q-2)(q+1)^2$ & $(q+1)(q^2-2q-1)/2+q^2$ \\ 
\hline
\end{tabular}

\begin{tabular}{|c||c|c|c|}
\hline
   $\cO$: &  $\cH_2$ & $\tS\in\tcS^1 $& $\tS\in\tcS^2 $ \\ 
   \hline
$|P^\perp \cap \cO|$: & $(q-1)(q+1)^2/2$ & $(q+1)(q^2-2q-1)/2+q^2 $ &$(q+1)(q^2-2q-1)/2$\\ 
\hline
\end{tabular}
\end{proposition}
\begin{proof}
Since points on $\cH_1$ are on no $\cQ_0$-tangent plane, by duality $P^\perp$ has no points in $\cQ_0$. Since $P^\perp$ must contain precisely one $\Sigma$-line, that must be external to $\cQ_0$.
On the other hand, each of the $q+1$ points of $P^\perp\cap\Tilde{\Sigma}$ lies on precisely $q+1$ $\Sigma$-planes tangent to $\cQ_0$. Since there are a total of $(q+1)^2$ of such planes, it follows that each of them contains precisely one point of $\Tilde{\Sigma}$. As a consequence, each such plane meets $P^\perp$ in precisely one point of $\Tilde{\Sigma}$. This implies that it meets $\Tilde{\cD}$ in $q-2$ points, explaining the case $\cO=\Tilde{\cD}$.
As usual, by looking at the intersections with the $\Sigma$-planes tangent to $\cQ_0$ we can explain the cases $\cO\in\{\cT,\cH_2\}\cup\tcS^2$, and the value for $\cO=\Tilde{\cQ}$ follows from the intersection properties of quadrics.
The value $|P^\perp\cap\tS|$, with $\tS\in\tcS^1$, it's obtained by double counting on the couples $(Q,P^\perp)$ with $Q\in\cH_2$ and $P\in\tS$, using Proposition \ref{S1 partial pts distribution} and the duality.
\end{proof}

\subsection{New Quasi-Hermitian surfaces}\label{sec: new quasi hermitian}

As a consequence of our investigation in Section \ref{section plane intersection numbers}, we can provide a new construction of quasi-Hermitian surfaces.
\begin{theorem}\label{quasi hermitian construction}If $S_1$ and $S_2$ are respectively in $\cS^1\cup\{\cH\setminus\cH_2\}$ and $\cS^2\cup\{\cH\setminus\cH_1\}$, then their join $S_1 \cup S_2$ is a quasi-Hermitian surface. If $S_1=S_{\xi_1}$ and $\hat{S}_1=S_{-\xi_1}$ and $S_2=S_{\xi_2}$ and $\hat{S}_2=S_{-\xi_2}$, then $S_1\cup S_2$ is isomorphic to $\hat{S_1}\cup\hat{S_2}$, in any other case, $S_1\cup S_2$ is not isomorphic to $\hat{S_1}\cup\hat{S_2}$.
\end{theorem}
\begin{proof}
Looking at the intersection numbers computed in Section \ref{section plane intersection numbers}, any plane intersects $\cS_1 \cup \cS_2$ in either $q^3+1$ or $q^3+q^2+1$ points. The last statement is a consequence of the fact that each such set is the join of orbits of the group $G_0$, and of Theorem \ref{pointorbits}.
\end{proof}
We also note that it is possible to obtain further quasi-Hermitian surfaces from the ones just constructed by applying the pivoting method on the tangent planes at points of $\cQ_0$. See \cite{SchillewaertVoorde2022} for further details on the pivoting method.
We finally address the isomorphism issue. To the best of our knowledge the only known constructions of quasi-Hermitian surfaces are those investigated in \cite{Pavese2015,LavrauwLiaPavese2023,Aguglia2013,AgugliaGiuzzi2022,AgugliaCossidenteKorchmaros2012,Napolitano2023}. For a quick overview of their properties, see \cite[Section 3]{LavrauwLiaPavese2023}. Since the group stabilizing the quasi-Hermitian surfaces constructed here is different from the one in \cite{LavrauwLiaPavese2023}, it is clear that the sets involved are not isomorphic. For the isomorphism issue with the other constructions, it is enough to look at the number of lines contained in the set passing through each point. See \cite[Section 3]{LavrauwLiaPavese2023} for a collection of the relevant results for the other constructions. A further construction, not mentioned there, takes $q^3+1$ pairwise skew lines, see \cite{Napolitano2023}. In each case, since points on $\cQ_0$ are incident with precisely $q+1$ lines contained in the quasi-Hermitian surface, it is immediate to see that the quasi-Hermitian surfaces are not isomorphic.
Further quasi-Hermitian surfaces can be obtained applying multiple times the pivoting construction of \cite{DS}. This has been investigated in details in \cite{SchillewaertVoorde2022}, see \cite[Theorem 5.2]{SchillewaertVoorde2022}. In particular, the authors showed that the pivoting method can be applied up to $q^2+1$ times on an Hermitian surface, pivoting at the tangent planes of the $q^2+1$ points of a given generator. 
Since the automorphism group of a quasi-Hermitian surface obtained from a Hermitian surface by pivoting is contained in the stabilizer of a plane (or of a line if it is multiple pivoting), and the group $G_0$ does not fix any plane, our construction can not be obtained by multiple pivoting from an Hermitian surface.
Summing up, we can state the following theorem.
\begin{theorem}
The quasi-Hermitian surfaces of Theorem \ref{quasi hermitian construction} are not isomorphic to any previously known construction.
\end{theorem}

\section{Equations for the quasi-Hermitian varieties arising from a Baer subgeometry}\label{section delta equation}
As pointed out in Proposition \ref{delta properties} and right after, the (reducible) surface of degree $2(q+1)$ of equation $H^2-4Q^{q+1}=0$ constitute an instance of the quasi-Hermitian surface constructed in \cite{Pavese2015}.
In this section we will show that, up to a change of projective frame, the quasi-Hermitian varieties of \cite{Pavese2015}, in any dimension, have a similar equation.

Let $\Sigma\simeq \PG(N, q)$ be a Baer subgeometry of $\PG(N, q^2)$ fixed by an involution $\phi$ and let $\cQ_0$ be a non-degenerate quadric of $\Sigma$. Let $\cL$ be the set of lines of $\PG(N, q^2)$ having $q+1$ points in common with $\Sigma$ and intersecting $\cQ_0$ in either one or $q+1$ points. Then 
\begin{align}
&  \bigcup_{\ell \in \cL} \ell \label{h3}
\end{align}  
is a quasi-Hermitian variety \cite{Pavese2015}. 

Let $\cB$ be such a quasi-Hermitian variety, and let $b(-,-)$ the bilinear non-degenerate quadratic form of $\Sigma$ that induces the quadratic form $Q$ (with $Q(u)=b(u,u)/2$) defining $\cQ_0$. 
The bilinear form $b(-,-)$ extends to $\PG(N,q^2)$, defining a quadric $\cQ$. Moreover, $b$ induces an Hermitian form $h(u,v)=b(u,v^\phi)$ as well. 
 By construction, the isotropic points of the Hermitian form $h$ constitute a non-degenerate Hermitian variety $\cH$ of $\PG(N,q^2)$, which is in permutable position with respect to $\cQ$.

Without loss of generality, we can assume $\Sigma$ to be the canonical subgeometry, and $\phi$ the Frobenius automorphism.
Assume that $P_0$ is a point of $\cB$. 
Then $\cB$ is given by the points on $\Sigma$-lines meeting $\cQ_0$ in one or $q+1$ points. For a given point $P$ in $\PG(N,q^2)\setminus\Sigma$, the unique $\Sigma$-line containing $P$ is the line $\langle P,P^\phi\rangle$. A point $P$ of $\PG(N,q^2)$ lies on precisely one $\Sigma$-line tangent to $\cQ_0$ if and only if it lies on precisely one $\Sigma$-plane tangent to $\cQ_0$. Therefore, $P\in\cB$ if and only if $P\in P_0^\perp$ and the line $\langle P,P^\phi\rangle$ has either $1$ or $q+1$ points on $\cQ_0$.

We claim that this happens if and only if 
\begin{equation}\label{equation pavese}
H(P)^2-4Q(P)^{q+1}=0
\end{equation}

Indeed, this holds true if and only if the equation
\[
b(P,P)x^2+(b(P,P^\phi)x+b(P,P)^q=0
\]
has precisely one solution.
If $P$ belongs to $\Sigma$, then $P$ clearly belongs to $\cB$, and it satisfies Equation \eqref{equation pavese}. Otherwise, there exists a unique $P_0\in\cQ_0$ such that $P_0\in\langle P,P^\phi\rangle$. Therefore, $P_0=P+xP^q$, with $x\in\F_{q^2}$ 
\begin{align*}
b(P_0,P_0)=&b(P+yP^q,P+yP^q)=b(P,P)+b(P,yP^q)+b(yP^q,P)+b(yP^q,yP^q)\\
&=2Q(P)+2xH(P)+2x^2Q(P)^q.
\end{align*}
Observe that we used $b(P^q,P^q)=2Q(P)^q$.
It follows that $P\in \PG(N,q^2)$ belongs to $\cB$ if and only if the equation
\[
Q(P)+xH(P)+x^2Q(P)^q=0
\]
has either precisely one solution, namely Equation \eqref{equation pavese} is satisfied, or it has precisely $q+1$ solutions. The latter case can only happen if it is identically zero as a polynomial in $x$, which means that both $Q(P)=0$ and $H(P)=0$, namely the point $P$ lies in the intersection of the commuting quadric and Hermitian varieties. This is equivalent to say that $P$ is on a generator.
Summarizing, we obtain the following Lemma.
\begin{lemma}
For any $n\geq 3$ there exist a non degenerate quadric $\cQ: Q=0$ and a non degenerate Hermitian variety $\cH:H=0$ such that the quasi-Hermitian variety of \cite{Pavese2015} has equation
\[
H^2-4Q^{q+1}=0.
\]
\end{lemma}

\section{Concluding remarks}

In this paper we studied a new model for threefold tensors, leading to a variety of results on semifields and a construction of new quasi-Hermitian surfaces. We believe this model can lead to further interesting geometric results for larger dimensions.
\subsubsection*{Declarations of interest: none}

\bigskip

{\footnotesize
\noindent\textit{Acknowledgments.}
The first author acknowledges the support by the Irish Research Council, grant n. GOIPD/2022/307. 
}

\bibliographystyle{plain}

\bibliography{Bibliography}

\section*{Appendix: Proof of the last part of Proposition \ref{S1 partial pts distribution}}\label{subsection end of proof}

The idea is to count, instead of the number of points of $\tS$ in $P^\perp$, with $P\in\tS\in\tcS^1$, the number of automorphisms of $G$ mapping $P$ in a point of $P^\perp$. This will lead us to solve a system involving the base locus of a pencil of Hermitian surfaces.

Without loss of generality, let $P\in\tS$ be of the form $(1,0,0,\xi)$, for some $\xi\in\F_{q^2}\setminus\F_q$. Observe that this is allowed because all the points of $\tS$ are on $\Sigma$-lines external to $\cQ_0$, and we can chose the line $\ell=(1,0,0,0)(0,0,0,1)$.
Therefore, $P^\perp$ has equation $x=\xi^qt$. We will compute the number $N$ of automorphisms of $G_0$ mapping $(1,0,0,\xi)$ into a point of $P^\perp$. Since the stabiliser $G_P$ of $P$ has size $q+1$, this number $N$ is also $(q+1)|P^\perp \cap \tS|$. 
$N$ is given by the number of solutions of the following equation:

\begin{align}
&xz^q+yt^q\xi-\xi_0^q(y^qt+x^qz\xi_i)=0;\label{baselocus}
\end{align}
under the further assumption that $x^{q+1}-y^{q+1}\neq 0$ and $z^{q+1}-t^{q+1}\neq 0$.
By looking at Equation \eqref{baselocus} as the base locus of a pencil of Hermitian surfaces, we obtain that $N$ is the number of solutions of the following system: 
\begin{align*}
    & H_1 : xz^q(1+\xi^{q+1})+x^qz(1+\xi^{q+1})+2yt^q\xi+2y^qt\xi^q=0; \\
    & H_2: (1-\xi^{q+1})(xz^q-x^qz)\theta=0; \\
    & H_3: x^{q+1}-y^{q+1}\neq 0; \\
    &H_4:  z^{q+1}-t^{q+1}\neq 0.
\end{align*}
Here $\theta$ is an element of $\Fqt$ with $\theta^q=-\theta$. We incorporate it in order to make $H_2$ a Hermitian equation.

To compute $N$ we will study the space of Hermitian matrices associated to these equations, obtaining $N$ via the inclusion-exclusion principle and the following result.
\begin{theorem}[Theorem 3.6, \cite{sheekeyThesis}]\label{sheekeythesis}
Let $U$ be a $d$-dimensional subspace of quadratic hermitian forms on $W$, where $W$ be a vector space $n$-dimensional over $\F_{q^2}$ . Let $A_r$ denote the number of elements of $U$ of rank $r$. Denote by $N_U$ the number of common zeros of the hermitian forms. Then
\[ N_U (0) =q^n+ \sum^{n-1}_{r=0}
(-1)^rA_r(q^{n-r}-(-1)^r).
\]
\end{theorem}
Let $A=(1+\xi^{q+1})$ and $B=(1-\xi^{q+1})\theta$. Note that, by the condition $P\in\tS$, $A,B\neq 0$ follows. The matrices associated to Hermitian forms $H_i$ are therefore the following:
\[
H_1:=\begin{pmatrix}
0 & 0 & A &  0 \\
0 & 0 & 0 & 2\xi \\
A & 0 & 0 & 0 \\
0 & 2\xi^q & 0 & 0 \\
\end{pmatrix};\quad
H_2:=\begin{pmatrix}
0 & 0 & B &  0 \\
0 & 0 & 0 & 0 \\
B^q & 0 & 0 & 0 \\
0 & 0 & 0 & 0 \\
\end{pmatrix};
\]
\[
H_3:=\begin{pmatrix}
0 & 0 & 0 &  0 \\
0 & 0 & 0 & 0 \\
0 & 0 & 1 & 0  \\
0 & 0 & 0 & -1 \\
\end{pmatrix};\quad
H_4:=\begin{pmatrix}
1 & 0 & 0 &  0 \\
0 & -1 & 0 & 0 \\
0 & 0 & 0 & 0 \\
0 & 0 & 0 & 0 \\
\end{pmatrix}.
\]
We apply now Theorem \ref{sheekeythesis}.
\begin{itemize}
\item In the subspace $\langle H_1,H_2\rangle$ there are precisely $q-1$ matrices of rank $2$ and $q(q-1)$ matrices of rank $4$.
Therefore,
\begin{align*}
    N_{\langle H_1,H_2\rangle}&=q^{8-2-0}+(q-1)q^{8-2-2}+q(q-1)q^{8-2-4}\\
    &=1+(q^2-1)(q^4+q^3+q^2+1)
\end{align*}.
\item In the subspace $\langle H_1,H_2,H_3\rangle$ there are precisely $2(q-1)$ matrices of rank $2$, $(q-1)^2$ matrices of rank $3$ and $q(q-1)^2+q(q-1)$ matrices of rank $4$. Therefore,
\begin{align*}
N_{\langle H_1,H_2,H_3\rangle}&=q^{8-3-0}+2(q-1)q^{8-3-2}-(q-1)^2q^{8-3-3}+q^2(q-1)q^{8-3-4}\\
&=1+(q^2-1)(q^3+2q^2+1)
\end{align*}.
\item In the subspace $\langle H_1,H_2,H_4\rangle$ there are precisely $2(q-1)$ matrices of rank $2$, $(q-1)^2$ matrices of rank $3$ and $q(q-1)^2+q(q-1)$ matrices of rank $4$. Therefore,
\begin{align*}
N_{\langle H_1,H_2,H_3\rangle}&=q^{8-3-0}+2(q-1)q^{8-3-2}-(q-1)^2q^{8-3-3}+q^2(q-1)q^{8-3-4}\\
&=1+(q^2-1)(q^3+2q^2+1).
\end{align*}
\end{itemize}
The subspace $\langle H_1,H_2,H_3,H_4\rangle$ needs a more detailed analysis.
\begin{itemize}
\item By the previous points, if one of $\gamma$ or $\delta$ is zero, there are precisely $3(q-1)$ matrices of rank $2$, $2(q-1)^2$ matrices of rank $3$ and $2q(q-1)^2+q(q-1)$ matrices of rank $4$.

\item If $\alpha,\gamma,\delta\neq0$ and $\beta=0$, $rank(\alpha H_1+\gamma H_3 +\delta H_4)=4$, unless $(\alpha,\gamma,\delta)=(\alpha,\gamma,\alpha^2 A^2/\gamma)$ or $(\alpha,\gamma,\delta)=(\alpha,\gamma,\alpha^2 \xi^{q+1}/\gamma)$. In the latter cases, the rank is at most $3$, and it is two if and only if $\alpha^2 A^2=\alpha^2 \xi^{q+1}$, which is not possible. Therefore, we have a total of $2(q-1)^2$ of such matrices of rank $3$ and $(q-1)^2(q-3)$ of rank $4$.
\item If $\alpha=\beta=0$ and $\gamma,\delta\neq 0$, then $rank(\beta H_2+\gamma H_3+ \delta H_4)=4$. This provides $(q-1)^2$ matrices.

\item If $\beta,\gamma,\delta\neq0$ and $\alpha=0$, then $rank(\beta H_2+\gamma H_3+ \delta H_4)=4$, unless $(\beta,\gamma,\delta)=(\beta,\gamma,\beta^2 B^{q+1}/\gamma)$. In the latter cases, the rank is precisely $3$. Therefore, we have a total of $(q-1)^2$ of such matrices of rank $3$ and $(q-1)^2(q-2)$ of rank $4$.

\item If $\alpha,\beta,\gamma,\delta\neq0$, then $rank(\alpha H_1+\beta H_2+\gamma H_3 +\delta H_4)=4$, unless $(\alpha,\beta,\gamma,\delta)=(\alpha,\beta,\gamma,(\alpha A+\beta B)(\alpha A+\beta B^q)/\gamma)$, or $(\alpha,\beta,\gamma,\delta)=(\alpha,\beta,\gamma,4\alpha^2\xi^{q+1}/\gamma)$.
In the latter cases, the rank is at most $3$, and it is two if and only if 
\[4\alpha^2 \xi^{q+1}=(\alpha A+\beta B)(\alpha A+\beta B^q).\] 
For any fixed value of $\beta\in\F_q^\times$, there are precisely zero solutions to this equation. As a consequence, we count $2(q-1)^3$ matrices of rank $3$ and $(q-1)^3(q-3)$ of rank $4$.

\end{itemize}
Summing up, we can now compute the following value:
\begin{align*}
&N_{\langle H_1,H_2,H_3,H_4\rangle}(0)=q^{8-4-0}+(3(q-1)q^{8-4-2}-(5(q-1)^2+2(q-1)^3)q^{8-4-3}+\\
&(2q(q-1)^2+q(q-1)+(q-1)^2(q-3)+(q-1)^2+(q-1)^2(q-2)+(q-1)^3(q-2))q^{8-4-4};
\end{align*}
namely
\[
N_{\langle H_1,H_2,H_3,H_4\rangle}(0)=1+2(q^2-1)(q+1).
\]
By the inclusion-exclusion principle, the number $N$ is given by
\[
N_{\langle H_1,H_2,H_3,H_4\rangle}(0)-N_{\langle H_1,H_2,H_3\rangle}(0)-N_{\langle H_1,H_2,H_4\rangle}(0)+N_{\langle H_1,H_2\rangle}(0);
\]
namely by
\[
(q^2-1)(q^3-3q-1).
\]

\end{document}